\documentclass[10pt]{article}

\usepackage[textsize=tiny]{todonotes}

\usepackage{arxiv}
\usepackage{amsthm} 
\theoremstyle{plain}
\newtheorem{theorem}{Theorem}[section]
\newtheorem{proposition}[theorem]{Proposition}

\theoremstyle{definition}

\newtheorem{remark}[theorem]{Remark}
\newtheorem{example}[theorem]{Example}

\theoremstyle{definition}

\usepackage{amsmath}             
\usepackage{amssymb}             

\usepackage{mathtools}      
\usepackage[scientific-notation=true]{siunitx}

\usepackage{booktabs}

\usepackage{float}
\floatstyle{plaintop}
\restylefloat{table}

\usepackage[width=.95\textwidth]{caption}
\usepackage[utf8]{inputenc} 
\usepackage{hyperref}       
\usepackage{url}            
\usepackage{nicefrac}       
\usepackage{graphicx}
\usepackage{doi}
\usepackage{bm}

\usepackage{enumitem}        
\usepackage{caption}         
\usepackage{subcaption}

\usepackage{algorithm}
\usepackage{algpseudocode}
\usepackage{pgf,tikz}
\usepackage{tcolorbox}

\newcommand{\apprxd}{\,{\buildrel d \over \approx}\,}
\newcommand{\equald}{\,{\buildrel d \over =}\,}

\newcommand{\ownTitle}{Generative modeling with low-rank Wasserstein polynomial chaos expansions}

\newcommand{\ownKeywords}{Generative modeling
Tensor train format, Wasserstein metric, polynomial chaos expansion, numerical upscaling, tensor chain, Sinkhorn divergence, optimal transport, multimodal distribution}

\hypersetup{
	breaklinks,
	colorlinks,
	linkcolor=gray,
	citecolor=gray,
	urlcolor=gray,
    pdftitle={Low-rank Wasserstein polynomial chaos},
    pdfsubject={math.NA, cs.LG},
    pdfauthor={Robert Gruhlke},
    pdfkeywords={low-rank, unsupervised learning, generative modeling, computational optimal transport, polynomial chaos expansion},
}

\begin{document}

\def\keywordname{{\bfseries \emph{Key words.}}}%

\def\mscclassname{{\bfseries \emph{AMS subject classifications.}}}%
\def\mscclasses#1{\par\addvspace\medskipamount{\rightskip=0pt plus1cm
\def\and{\ifhmode\unskip\nobreak\fi\ $\cdot$
}\noindent\mscclassname\enspace\ignorespaces#1\par}}
\def\codename{{\bfseries \emph{Code.}}}%
\def\code#1{\par\addvspace\medskipamount{\rightskip=0pt plus1cm\noindent\codename\enspace\ignorespaces\url{#1}\par}}

\title{\ownTitle}

\author{\hspace{1mm}\textcolor{black}{Robert Gruhlke} \\
	Weierstrass Institute\\
	Berlin, Germany \\
	\texttt{r.gruhlke@fu-berlin.de} \\
	\And
Martin Eigel \\
     Weierstrass Institute \\
     Berlin, Germany \\
 	\texttt{eigel@wias-berlin.de} \\
}

\maketitle

\begin{abstract}
A new \textit{Wasserstein multi-element polynomial chaos expansion} (WPCE) is proposed, which is inspired by recent advances in computational optimal transport for estimating Wasserstein distances.
The developed method combines unsupervised learning with the explicit functional representation of a random vector $Y$.
Its training only relies on a finite set of samples from an unknown distribution, which is used to minimize a regularized empirical Wasserstein metric known as debiased Sinkhorn divergence.
An interesting application that motivates the approach comes from the numerical upscaling of non-periodic random fields.
These appear for instance in materials with stochastic micro-structure.
Computations with such data is computationally challenging and requires some form of model reduction.
The WPCE can be used to obtain a macroscopic random material that retains the properties of the microscopic material on a larger scale, i.e. as integrated properties on subdomains.
In contrast to established methods from stochastic homogenization, the WPCE allows to not only represent the constant characterization of the effective material but also contains higher order stochastic information about the effective behavior.

A striking feature of the new method is the generalization of common diffeomorphic transport maps to the case of discontinuous and non-injective model classes $\mathcal{M}$ with possibly different input and output dimension.
It computes a (functional) relation $Y=\mathcal{M}(X)$ in distribution with input random variables $X$ and target $Y$, approximating the given sample distribution.
For the underlying polynomial chaos expansion (PCE) of $\mathcal M$, a stochastic coordinate system $X$ is assumed.
Since the used PCE grows exponentially in the number of random coordinates of $X$, a new low-rank format given as stacks of tensor trains (STT) is introduced.
This alleviates the curse of dimensionality, leading to only linear dependence on the input dimension.

By the choice of the model class $\mathcal{M}$ and the smooth loss function, higher-order optimization schemes and in particular Riemannian optimization become possible.
The proposed approach is illustrated numerically with a high-dimensional upscaling problem, which considers a microscopic random non-periodic composite material.
It results in a computationally tractable effective macroscopic random field in adapted stochastic coordinates.
By using a relaxation to a discontinuous model class, multimodal distributions also become tractable.
\end{abstract}
   \keywords{\ownKeywords}
    \vspace{-1em}
    \mscclasses{15A69 \and 41A30 \and 62J02 \and 65Y20 \and 68Q25}
    \vspace{-1em}

\section{Motivation}
\label{sec:motivation}

Measure transport has become a popular research topic in many scientific fields and is in particular of great interest in Uncertainty Quantification (UQ) and modern (scientific) Machine Learning (ML).
Our contribution is a new strategy to obtain a functional model that approximates the law of the target $Y$ in distribution with values in $\mathbb{R}^N$ for which only samples are available.
This computed model is given explicitly in functional form as expansion in a compressed multi-element polynomial basis~\cite{wan2006multi} in a \textit{stochastic reference coordinate system} determined by a random variable $X$ with values in $\mathbb{R}^M$.
We denote this model class by $\mathcal M=\mathcal{M}(X)$ and optimize the representation by means of \textit{computational optimal transport}~\cite{peyre2019computational} using measure fitting with a \textit{debiased Sinkhorn loss}~\cite{feydy2019interpolating}.
The choice of the reference coordinates $X$, consisting of independent random variables, ideally is adapted to the problem at hand\footnote{However, this interesting topic is not examined in this work and we assume a given reference coordinate system.
An approach in this direction can be found in~\cite{soize2010identification} using a log-loss to determine the needed input dimension and the associated numbers of degree of freedom.} but an ``uninformed'' choice such as standard Gaussian is always possible.

This leads to a polynomial chaos expansion of the form
\begin{equation}
\label{eq:model-intro}
Y\apprxd Y_\text{PCE}:=\mathcal{M}(X)=\sum_{s=1}^{S}\sum_{\alpha\in\mathbb N_0^M}C_s[\alpha]P_\alpha^s(X),
\end{equation}
which approximates the exact but unknown $Y$ (in distribution), where $S\in\mathbb{N}$ is the number of multi-elements and $P_\alpha^s$ denotes the multi-element polynomial chaos indexed by multiindex $\alpha$ on a decomposition of $\mathbb R^M$ into $S$ subdomains.

In case that highly nonlinear maps are required to accurately approximate the transformation of the reference $X$ to the target distribution represented via $Y$, the functional representation easily leads to very high-dimensional representations.
In particular, the complexity scales exponentially (``curse of dimensionality'') in the number of stochastic coordinates $M$, determining the size of the coefficient tensor $C_s$.
To make this approach feasible in practice, it is inevitable to compress the coefficient.
We tackle this challenging task by introducing \textit{stacks of low-rank tensor trains} (STT) format, which is an extension of the popular tensor trains (TT)~\cite{oseledets2011tensor}.
These have been used successfully e.g. in UQ applications~\cite{ENSW19,EMPS20,dolgov2015polynomial,dolgov2019TTdensities}, quantum physics models~\cite{werner2016positive} and quantum chemistry~\cite{szalay2015tensor}.
In fact, the described STT format -- opposite to the standard TT format -- is capable to represent vector valued output.
It leads to a favorable overall complexity that again scales only linearly in $M$.

The model design~\eqref{eq:model-intro} can be seen as a generalization of the (much stricter) notion of optimal transport (OT), which in our method is conceptually carried out on each part of a decomposition of the preimage, mapping to a selection of samples determining the image.
This approach allows for multimodal distribution representations in terms of random variables. 
While OT requires a diffeomorphism between two spaces of equal dimension, this assumption is relaxed in our approach, in particular allowing for non-continuous probability mass transport.
Such a relaxation is inevitable if multimodalities should be represented with high accuracy.
Moreover, we do not require invertible mappings $\mathcal{M}$ and use the notion of convergence in distribution for a generalized functional representation where input dimension $M$ and output dimension $N$ may differ.

Our methodological contribution combines different tools, ranging from optimal transport theory to polynomial chaos representations of random variables.
It can be understood as an alternative model class to generative adversarial neuronal networks (GAN) and its extensions to Wasserstein distances denoted as Wasserstein GANs (WGAN), allowing for Riemannian optimization schemes.
We illustrate the performance of the technique by solving several challenging problems in uncertainty quantification such as the representation of multimodal distributions or random fields defined on different spatial scales, connected by the notion of ``stochastic numerical upscaling''. 
Our contribution is related to different thematic fields, summarized in the following.

\paragraph{\textbf{Optimal transport}}

The notion of optimal transport~\cite{villani2009optimal} allows to compare probability measures in terms of required workload or costs to move one probability mass to another~\cite{kantorovich1942transfer}. 
In the standard setting of the Monge problem~\cite{villani2009optimal}, measureable sets $X$ and $Y$ with respective measures $\mu$ and $\nu$ such that $\mu(X)=\nu(Y)<\infty$ are assumed.
The task is to find an injective transport mapping $T:X\to Y$ subject to some cost function $c:X\times Y\to\mathbb R_+$,
\[ T = {\arg\min}_\tau \left\{\int_X c(x,\tau(x))\mathrm{d}\mu(x) : \tau_\#\mu=\nu \right\},
\]
such that
\[
\int_X g(T(x))\mathrm{d}\mu(x) = \int_Y g(y)\mathrm{d}T_\#\mu(y)
\]
for any $\mu$-measurable $g:X\to\mathbb R$.
Let $Y\sim\mu$ and $X\sim\nu$.
If such a $T$ exists, then 
\begin{equation}
\label{eq:transportmap}
    Y \equald T(X).
\end{equation}
This concept of a transport map has been examined thoroughly in the context of Bayesian inverse problems~\cite{morrison2017beyond,baptista2020adaptive,marzouk2016introduction,brennan2020greedy,detommaso2018stein}.
As a matter of fact, these problems consist of determining a posterior measure given a prior measure conditioned on a set of observations, which resembles the transport problem.
Consequently, the knowledge of the (approximate) transport map can help to significantly alleviate the high computational burden that e.g. is typical for Markov chain Monte Carlo methods.
With this in mind, functional representations in a polynomial basis exploiting the beneficial structure of the Knothe-Rosenblatt transform were for instance developed in~\cite{spantini2018inference,cui2021conditional,cui2021deep}.
For the formulation of the variational problem, the Kullback-Leibler divergence is used. However, while this type of loss functional is appropriate in the setting of Bayesian inference due to the absolute continuity of the prior and posterior measures, the latter property does in general not hold true when optimizing with respect to measures.

Another drawback is that the existence of an injective transport map $T$ is not ensured in general.
A relaxation of this problem was formulated by Kantorovich, avoiding the missing guaranteed existence of a map $T$ in the Monge problem.
This concept is discussed in more detail in the first section of this work, including its computational challenges.
By introducing a model class $\mathcal{M}$ and a reference coordinate system represented by a random variable $X$, we obtain a representation motivated by~\eqref{eq:transportmap} given by
$$
Y \equald \mathcal{M}(X).
$$
This type of representation generalizes the notion of transport in the following way.
First, if $X$ and $Y$ are random vectors with values in $\mathbb{R}^M$ and $\mathbb{R}^N$, $N,M\in\mathbb{N}$, then we allow that $M\neq N$ while in a classical transport formulation $M=N$ has to be satisfied.
Second, we can drop the common injectivity or even diffeomorphism assumptions on $T$.
Moreover, we allow $\mathcal{M}$ to be discontinuous, a feature that enables a larger expressivity, which is essential for multimodal random variables $Y$.

\paragraph{\textbf{Polynomial chaos \& low-rank tensor formats}}
Polynomial chaos (PC) or multi-element PC~\cite{wan2006multi} is ubiquitous in UQ~\cite{le2010spectral,schwab2011sparse,cohen2015approximation} to represent a random variable $Y\colon \Omega\to\mathbb{R}^N$ with finite variance in a coordinate system of $M$ independent random variables $X=(X_1,\ldots,X_M)$.
These are used for an expansion in polynomials orthogonal with respect to given distribution of $X$.
It is an explicit functional representation of the form
\[
Y(\omega) = Y(X(\omega)) = \sum_{\alpha\in\mathbb N_0^M} C[\alpha]P_\alpha(X(\omega)),\quad\omega\in\Omega.
\]
This surrogate includes all statistical information and allows for a computationally efficient evaluation of statistical quantities.
In UQ, it has become decidedly popular for the representation of data and solutions of random differential equations, especially PDEs~\cite{ernst2012convergence,schwab2011sparse,cohen2015approximation}.
Such polynomial chaos surrogates can be obtained by different means, the most common of which are stochastic Galerkin methods and empirical least squares projections, see~\cite{beck2014convergence,blatman2008sparse,le2010spectral,eigel2014adaptive} and~\cite{chkifa2015discrete,migliorati2013approximation,eigel2019variational,eigel2019non}, respectively.
The convergence of the polynomial chaos expansion follows from the classical theory of Cameron and Martin~\cite{ernst2012convergence}.

When $N=1$, low-rank tensor formats allow to compress the high-dimensional representation.
These formats have initially been devised in the quantum chemistry and physics community and have later been reinvented in the field of computational mathematics.
We refer the interested reader to the monograph~\cite{hackbusch2012tensor} and the overview articles~\cite{nouy2015low,bachmayr2016tensor},

The TT format was (re-)introduced in~\cite{oseledets2011tensor} and has become quite popular in computational mathematics, yielding a compressed representation of $C$ for $N=1$ by 
\begin{equation}
\label{eq:classicalTT}
C[\alpha] = C[\alpha_1, \ldots, \alpha_M] =
C_1[\alpha_1]\cdot\ldots \cdot C_m[\alpha_m]
\end{equation}
with order $3$ tensors $C_m\in \mathbb{R}^{r_{m-1},d_m, r_m}$, tensor train rank $r=(r_1,\ldots,r_{m-1})\in\mathbb{N}^{M-1}$, $r_0=r_m=1$ and polynomial degree $d_m\in\mathbb{N}$ in coordinate $m$.
The number of degrees of freedom in this format is bounded by $\mathcal O(Mdr^2)$ with $r=\max\{r_1,\ldots,r_M\}, d=\max\{d_1,\ldots,d_M\}$.
The complexity of this hierarchical representation thus is linear in the stochastic dimensions $M$.
Moreover, these tensors form a differentiable manifold such that e.g. Riemannian optimization methods become feasible~\cite{steinlechner2016riemannian}.
For $N>1$, as mentioned above these formats cannot be used directly and thus require a modified design.

One idea would be to decompose $C = C[i,\alpha]$ with $i=1,\ldots,N$ denoting the $i$-th output component, as in~\eqref{eq:classicalTT} but using forth order tensor $C_1=C_1[i,\alpha_1]$ and third order tensors $C_m=C_m[\alpha_m]$ for $m=2,\ldots,M$.
However, in our application in practice this format results in an exponential growth of ranks with respect to the output dimension $N$.
Alternatively, we extend the standard TT format to the $N\geq 1$ case, leading to a problem-adapted \textit{stack of tensor trains} (STT) format, allowing each component tensor $C_m$ to interact with any output component $i=1,\ldots,N$.
The STT format reads
\begin{equation}
\label{eq:ownformat}
    C[i, \alpha]= C[i, \alpha_1,\ldots,\alpha_M]= C_1[i, \alpha_1]\cdot\ldots \cdot C_M[i, \alpha_M],\qquad i=1,\ldots,N,
\end{equation}
where all component tensors $C_m\in \mathbb{R}^{N, r_{m-1},d_m, r_m}$ are of order 4.
For fixed $i=1,\ldots,N$,~\eqref{eq:ownformat} can be interpreted as a classical tensor train format.  

For the bigger picture, it should be noted that the graph structure of tensor formats in some way resembles the topology (and expressiveness) of neural networks (NN)~\cite{cohen2016expressive,cohen2016convolutional,ali2020approximation,ali2021approximation,bachmayr2021approximation}.
In fact, tensor networks can be seen as a subset of NNs with somewhat lesser expressivity\footnote{mainly because they are a multilinear instead of a nonlinear representation, but for standard regularity spaces they exhibit the same representation complexity} on the one hand, but much richer mathematical structure on the other.
Moreover, since the classical tensor train format of fixed rank forms a Riemannian manifold~\cite{holtz2012manifolds}, the same holds for the Cartesian tensor product of tensor trains as given through the STT format.

\textbf{Relation to Generative Adversarial Neural Networks (GAN)}
A very popular generative model represented by neural networks (NN) is the notion of GANs~\cite{goodfellow2014generative}.
We introduce basic ideas since these methods pursue a similar goal as our approach but differ fundamentally in the way they achieve it.
Classical Wasserstein GANs (WGANs), \textit{e.g.} \cite{liu2019wasserstein} are based on the dual formulation of $\mathcal{W}_1$ optimal transport based on the difference of expectations of $1$-Lipschitz functions given by
$$
\min\limits_{\theta} \max\limits_{f:\operatorname{Lip}(f)\leq 1} \mathbb{E}_{X\sim\mu}[f(X)] - \mathbb{E}_{X\sim\nu[\theta]} [f(X)].
$$
The parametric measure $\nu[\theta]$ is determined by some known distribution $\nu$ represented as NN model $g[\theta]$ subject to parameters $\theta$, cf.~\cite{arjovsky2017wasserstein}.
Accordingly, the following GAN optimization has to be solved, 
\begin{equation}
\label{eq:minmaxformalism}
\min\limits_{\theta} \max\limits_{f:\operatorname{Lip}(f)\leq 1} \mathbb{E}_{X\sim\mu}[f(X)] - \mathbb{E}_{X\sim\nu} [f(g[\theta](X))].
\end{equation}
It has been observed that WGANs are difficult to train, in particular due to exhibited non-robustness.
They may perform unconvincingly, e.g. when multi-modal distributions are approximated as in~\cite{liu2019wasserstein}.
Since classical WGANs with $\ell_1$ cost may not converge, several alternative cost functions were introduced such as the $\ell_2$ cost in~\cite{pinetz2018optimized}.
Smoothed WGANs based on a regularized Sinkhorn loss where introduced in~\cite{genevay2018learning} and applied for the pictures data sets MNIST and CIFAR-10 in~\cite{sanjabi2018convergence}.
Furthermore, the work in~\cite{behrmann2019invertible} considered the construction of invertible residual networks as generative models.
In this construction, the involved Lipschitz constraints to ensure invertibility remain a challenging task in the actual application.

Opposite to the GAN formulation, we do not use the dual characterization and also do not require a discriminator approximation or Lipschitz constraints.
The used explicit polynomial chaos model class corresponds to the \textit{generator} only within the GAN min-max formulation \eqref{eq:minmaxformalism}.
This leads to a single model that needs to be trained.
However, we underline the still present challenge of non-convex and non-linear optimization involving local minima, usually resulting in inaccurate solutions.
The effects already occur in simple scenarios to be discussed in examples in due course in the presentation of our approach.

We rely on second order optimization schemes based on automatic differentiation that become feasible because of the smooth parameter dependence of the debiased Sinkhorn loss.
Additionally, we consider the UQ point of view and interpret this technique as a change of coordinates, which allows for possible stochastic dimensional reduction when $M<N$.

\paragraph{\textbf{Application context}}
This work is inspired considerably by the works~\cite{soize2010identification} and~\cite{sarfaraz2018stochastic} where the derivation of effective random coarse grained fields is examined given only limited fine scale information in terms of samples or observations.
In~\cite{soize2010identification} the construction of a random field representation in terms of a polynomial chaos expansion (PCE) was considered with coefficients defined on a Stiefel manifold associated to inherited correlation structures in the underlying Kosambi–Karhunen–Loève expansions.
In order to make this approach tractable, a simplified loss function -- namely a tensorized approximation of a maximum liklihood estimation -- using one dimensional kernel density estimates to gain statistical information was considered. 
In contrast to this approach, here we consider a PCE obtained from optimizing (approximations of) Wasserstein losses.
We denote this technique \textit{Wasserstein Polynomial Chaos Expansion} (WPCE).
It allows to capture the sample distribution accurately in an unsupervised manner without relying on kernel density estimations that converge slowly and lose statistical information in the estimation process.

The work~\cite{sarfaraz2018stochastic} introduces a stochastic Bayesian upscaling framework to obtain PCEs of random variables.
There, the representation is obtained via a projection in terms of conditional expectations given as spectral Kalman filter.
In its most elaborate form, this approach yields an assimilated random variable that matches the empirical distribution in the first and second moment only.
Nevertheless, it should be noted that extensions to higher moments are mentioned.
The practical realization however would be rather challenging.

A further strong motivation for the presented approach is provided by the applications we have in mind.
First, an efficient generation of samples from a highly nonlinear (i.e. non-Gaussian) target measure is a challenging task, which can be found in many statistical applications.
Getting computational access to (e.g. multimodal) randomness using standard methods such as Markov Chain Monte Carlo is difficult and requires a prior reconstruction of densities from samples, which is elaborate and only exhibits slow convergence.
Our approach directly yields such representations in terms of \textit{functional approximations of random variables}.
The availability of a fully functional representation of the observed randomness exhibits two obvious advantages.
On the one hand, it can be used in subsequent computations as a closed representation of randomness, relying on a known underlying reference coordinate system $X$ and in turn can be used as a generator (or surrogate) to produce more samples of $Y$.
On the other hand, we are interested in the \textit{numerical upscaling} of random non-periodic micro-structures as discussed in Section~\ref{sec:upscaling}.
In contrast to classical (random) homogenization with constant effective coefficient~\cite{blanc2007stochastic,gloria2012optimal,gloria2014optimal}, we strive for an approach where the \textit{macroscopic representation remains a random field}, hence also containing statistical information (of the finer scale) on the coarser scale.
For practical applications, this provides much more useful information since e.g. the variance (or other statistical quantities of interest) of some system response can be determined.
Upscaling of stochastic non-periodic material is carried out in a stochastic pointwise sense, i.e. a sample of the fine scale field is upscaled numerically to obtain a sample of the related coarse scale random field for a prescribed domain decomposition.
Note that the homogenized material and the microscale material both have unknown or intractable distributions.
In the homogenized framework, we allow for stochastic fluctuations without the need to (approximate) a stochastic homogenization problem defined on an unbounded domain.

The structure of this paper is as follows.
The next Section~\ref{sec:computational optimal transport} provides an overview of recent advances in computational optimal transport, building the framework for our approach.
Section~\ref{sec:modelclass} introduces the structure of our model class based on a compressed multi-element PCE, its underlying reference coordinate system $X$ and discusses challenges in the related optimization problem.
Section~\ref{sec:experiments} is devoted to the investigation of the numerical performance of the proposed method in the setting of multi-modal randomness and the representation of random fields on possible different scales in terms of stochastic upscaling.
Eventually, a conclusion in Section~\ref{sec:conclusion} summarizes the main achievements and further directions.

\section{Computational optimal transport}
\label{sec:computational optimal transport}

This section is concerned with the central goal of this work, namely the representation of a random variable $Y$ with image in $\mathbb{R}^N$, which is unknown a priori and only finitely many samples are available or can be generated, typically involving high computational costs.
As a standard tool, \textit{kernel density estimation} (KDE) \cite{chen2017tutorial} is applied on an ensemble of samples to approximate a -- possibly not existing -- density of $Y$.
Once a representation is obtained, further samples of $Y$ can be drawn based on the density information with potentially negligible effort.
However, the KDE suffers from the curse of dimensionality (CoD), \textit{i.e.} the number of samples required for a reasonable approximation of the underlying density grows exponentially in the dimension $N$. 

As an alternative strategy, we aim for a functional representation
\begin{equation}
\label{eq:modelclassrepresentation}
    Y \equald \mathcal{M}(X)
\end{equation}
based on some model class $\mathcal{M}$ and some a priori chosen stochastic coordinate system encoded in a random variable $X$.
We refer to Section~\ref{sec:modelclass} for a detailed discussion.
This representation yields several advantages.
First, it does not require the existence of an underlying density of $Y$.
Second, samples of $Y$ can be drawn with little effort if obtaining samples from $X$ and the propagation through $\mathcal{M}$ can be carried out efficiently.
Finally, this functional representation can be advantageous to compute integral quantities in an analytical (sampling-free) manner.
A common requirement is for instance the evaluation of quantities of interests like moments.
This also is a central ingredient in stochastic Galerkin schemes~\cite{MR1083354,matthies2008stochastic,le2010spectral} for the computation of inner products. 

Given only sample information of $Y$ the Wasserstein distance is a promising tool to compare distributions or their empiricial counterparts even in the case when the measures do not share support.
This property is useful when optimizing over $\mathcal{M}$ based on a random or non-informed start value.
In what follows, we give an overview of results and challenges related to the Wasserstein distance and its computational advances based on regularization that in the end allows to minimize the distance of a model in distribution, which is paramount to our method.

\subsection{Exact optimal transport: Wasserstein distance}
\label{sec:exacttransport}
Since only samples of $Y$ are available, $Y$ and $\mathcal{M}(X)$ can only be compared in terms of samples or equivalently by determining the distance of the related empirical measures defined through the finite samples of $Y$ and $\mathcal{M}(X)$.
We intend to compute the distance of measures with the help of the \textit{Wasserstein} or \textit{Kantorovich–Rubinstein} metric \cite{villani2009optimal}, which are introduced in the following.

Let $(\mathcal{V}, \mathrm{d})$ be a complete metric space, $c\colon \mathcal{V}\times\mathcal{V}\to \mathbb{R}$ a symmetric continuous cost function and let $\mathcal{D}(\mathcal{V})$ denote the set of probability measures on $\mathcal{V}$.
The Kantorovich formulation~\cite{kantorovich1942transfer} of optimal transport costs or Wasserstein costs between probability measures $\mu,\nu \in \mathcal{D}(\mathcal{V})$ is defined as the minimal cost required to move each element of mass of $\mu$ to each element of mass of $\nu$ written as a linear problem over the set of transportation plans, which are probability measures on $\mathcal{V}\times\mathcal{V}$, namely
\begin{align}
   \mathcal{W}(\mu,\nu) := \mathcal{W}_c(\mu,\nu) := \operatorname{inf}\limits_{\pi\in\mathcal{D}(\mathcal{V}\times\mathcal{V})} \left\{ \langle c, \pi \rangle \,:\, \pi_1 = \mu, \pi_2 = \nu\right\}, 
\end{align}
where $\pi_1 = \int\limits_{y\in\mathcal{X}} \mathrm{d}\pi(\cdot,y)$ and $\pi_2 = \int\limits_{x\in\mathcal{X}} \mathrm{d}\pi(x,\cdot)$ are the marginals of the transportation plan $\pi$. 

For $p\in[1,\infty]$ let $c(v_1, v_2) = \mathrm{d}(v_1,v_2)^p$ and $\mathcal{D}_p(\mathcal{V})$ denotes the set of measures on $\mathcal{V}$ with finite moment of order $p$.
Define the $p$-th \textit{Wasserstein distance} by 
\begin{equation}
\label{eq:pWasserstein}
    \mathcal{W}_p(\mu,\nu) := \mathcal{W}_{c^p}(\mu,\nu)^{1/p},\quad \mu,\nu\in\mathcal{D}_p(\mathcal{V}).
\end{equation}
In practice, we consider the measure $\mu$ to be unknown and only $n\in\mathbb{N}$ \textit{iid} samples of $\mu$ or equivalently an empirical measure $\mu_n$ are available.
This raises the question of how well such an empirical measure explains the distribution $\mu$ in terms of the Wasserstein distance~\eqref{eq:pWasserstein}.
Since it metricizes the weak convergence of measures~\cite{villani2009optimal} and the empirical measure converges weakly to $\mu$~\cite{varadarajan1958convergence}, it follows for $p\in[1,\infty)$ that
\begin{equation}
    \mathcal{W}_p(\mu,\mu_n) \to 0,\quad \mu \text{-a.s.  as } n\to\infty, 
\end{equation}
provided that $X$ is compact and separable and $\mu$ is a Borel measure.
Unfortunately, the convergence rate of $\mu_n$ to $\mu$ in Wasserstein distance inevitably suffers from the curse of dimensionality.
A negative result~\cite{dudley1969speed} showed that if $\mu$ is absolutely continuous with respect to the Lebesgue measure on $\mathbb{R}^N$ then for some $C>0$ it holds
$$
\mathbb{E}[\mathcal{W}_1(\mu,\mu_n)] \geq C n^{-1/N}.
$$
An upper bound can be obtained if $N>2$ and $\mu$ are compactly supported on $\mathbb{R}^N$.
Then, 
$$
\mathbb{E}[\mathcal{W}_1(\mu,\mu_n)] \leq C n^{-1/N}.
$$
Similar negative results have been extended to the case $p\in[1,\infty]$ in~\cite{weed2019sharp} based on the concept of upper and lower Wasserstein dimensions of the support of the underlying measure.
In the special case of measures $\mu$ with regular $N$-dimensional support~\cite{mattila1999geometry} that is absolute continuous with respect to a Hausdorff measure, the same upper and lower asymptotical bounds hold true for $\mathbb{E}[\mathcal{W}_p(\mu,\mu_n)]$ at least with $p\in[1,N/2]$.

A much more involved analysis in~\cite{weed2019sharp} shows that the rate of convergence of $\mathcal{W}_p(\mu,\mu_n)$ depends on a notion of the intrinsic dimension of the measure $\mu$, which can be significantly smaller than the dimension of the metric space on which $\mu$ is defined.
Moreover, in the finite-sample regime, wildly different convergence rates may appear.
In particular, they can enjoy much faster convergence for small $n$.
An exemplary phenomenon is based on the different dimensional structure of measures at different scales, i.e. the so called \textit{multi-scale behaviour}.
For illustration, let $\mu$ be $(m, \Delta)$-clusterable, i.e. $\operatorname{supp}(\mu)$ lies in the union of $m$ balls of radius at most $\Delta$.
Then for all $n\leq m(2\Delta)^{-2p}$,
$$
\mathbb{E}[\mathcal{W}_p^p(\mu,\mu_n)]\leq (9^p+3)\sqrt{\frac{m}{n}}.
$$
Notably, $\mu_n$ convergences to $\nu$ as $n^{-1/{2p}}$ in the pre-asymptotic regime.
This result applies for example to mixtures of Gaussian distributions.
If the measure $\mu$ with support in $\mathbb{R}^N$ has approximately low-dimensional support for dimension $\underline{N} < N$ then one obtains $\mathbb{E}[\mathcal{W}_p^p(\mu,\mu_n]\leq C n^{-p/\underline{N}}$ in the finite sample range.
Note that this finite range convergence may be much faster than $n^{-p/N}$ provided $\underline{N}\ll N$.

So far we only discussed bounds on the expectation of $\mathcal{W}_p(\mu,\mu_n)$.
This analysis is motivated by the so called \textit{concentration around expectation} property.
In particular, for $n\in\mathbb{N}$ and $p\in[1,\infty)$ McDiarmid's inequality yields~\cite{weed2019sharp}
$$
\mathbb{P}\left( \mathcal{W}_p^p(\mu,\mu_n) \geq 
\mathbb{E}[\mathcal{W}_p^p(\mu,\mu_n)] + t\right) \leq \exp\left(-2n t^2\right),\quad t\in\mathbb{R}.
$$
We may now consider two random variables $Y\sim\mu$ and $\mathcal{M}(X)\sim \nu$ with values in $(\mathbb{R}^N,\|\cdot\|)$, \textit{iid} samples $Y_1,\ldots,Y_n\sim\mu$, $\mathcal{M}(X_1),\ldots, \mathcal{M}(X_m)$ and denote by $\mu_n$ and $\nu_m$ the empirical measures of $\mu$ and $\nu$, respectively, for given $n,m\in\mathbb{N}$ defined by 
$$
 \mu_n := \sum\limits_{i=1}^n a_i\partial_{Y_i},\quad  \nu_m := \sum\limits_{k=1}^m b_j \partial_{\mathcal{M}(X_j)}.
$$
Here, $\partial_x$ denotes the delta distribution in point $x$ and $a_i\equiv 1/n$ and $b_j\equiv 1/m$. 
Let $p\leq0$ and $C_p = [\|Y_i-\mathcal{M}(X_j)\|^p]_{ij}\in\mathbb{R}^{n,m}$ be the cost matrix and $\bm{1}_n, \bm{1}_m$ be the $n$ or $m$-dimensional vector of ones and define $a=(a_i)$ and $b=(b_j)$.
Then, the Wasserstein distance $\mathcal{W}_p(\mu_n, \nu_m)$ is characterized by
\begin{align}
\label{eq:discreteOT}
    \mathcal{W}_p^p(\mu_n,\nu_n) = \min\limits_{\pi \in \mathcal{U}_{nm}} \langle \pi, C_p\rangle,\quad \mathcal{U}_{nm} = \{ \pi \in \mathbb{R}_+^{n,m}\,:\, \pi \bm{1}_m = \bm{1}_n/m, \pi^T\bm{1}_n = \bm{1}_m/m\}.
\end{align}
This is a classical transport problem with the special case of uniform weights.
It can thus be solved by standard solvers for min-cost flow problems.
Unfortunately, the complexity depends at best cubically on the input data.
Hence, when $n$ or $m$ get large, this approach becomes impractical.
The optimal transport plan $\pi^\ast$ is sparse \cite{dvurechensky2018computational} and lies on the boundary of the feasible region, which is a polytope~\cite{peyre2019computational}.
Consequently, the solution process can become unstable due to ambiguities or discontinuous jumps of the cost functional and its orientation onto the spanned polytope of the admissible region, which is caused by the enforced constraints due to the admissable set $\mathcal{U}_{nm}$.
Moreover, given an optimal transport plan $\pi^\ast$ and a parametric cost function $c=c(\theta)$, which is for example induced by the parameters of the model class $\mathcal{M}$, then $\theta \mapsto \pi^\ast (c(\theta))$ is not differentiable.

\subsection{Inexact optimal transport: Entropic regularization}
\label{sec:inexact optimal transport}

Motivated by the drawbacks of the exact optimal transport presented in Section~\ref{sec:exacttransport}, namely a lack of regularity and intractable computational costs for large sample sizes, a new era to computational optimal transport has started with~\cite{cuturi2013sinkhorn}.
It is based on \textit{entropic regularization} of the Wasserstein distance, which can be traced back to the work of Schrödinger~\cite{schrodinger1932theorie} and reads
\begin{align}
\label{eq:entropicLoss}
\mathcal{W}_{c,\epsilon}(\mu,\nu) := \min\limits_{\pi\in\mathcal{D}(\mathcal{X}\times\mathcal{X})}\langle \pi, c\rangle +   \epsilon
    \operatorname{KL}(\pi,\mu\otimes\nu) \\
    \text{subject to } \quad \pi\geq 0, \quad \pi_1=\mu \quad \pi_2=\nu,
    \end{align}
    for positive values of ``temperature'' $\epsilon > 0 $ with \textit{entropic penalty} (or \textit{Kullback-Leibler divergence})
    \begin{align}
        \operatorname{KL}(\pi,\mu\otimes\nu) := \left\langle \pi, \log\frac{\mathrm{d}\pi}{\mathrm{d}\mu\otimes\nu}\right\rangle - \langle \pi, 1\rangle + \langle \mu\otimes\nu , 1\rangle.
    \end{align}
While the $p$-th Wasserstein distance satisfies the triangle inequality based on the positive definite loss function $c(x,y) = \frac{1}{p}\|x-y\|^p$, those geometric properties do not hold for the entropic loss $\mathcal{W}_{c,\epsilon}$. 
Therefore, minimizing an $\mathcal{W}_{c,\epsilon}$ loss is not a sensible approach~\cite{feydy2020geometric}.
In general, if $\nu$ is a given target measure, there exists a degenerate measure $\mu$ such that $\mathcal{W}_{c,\epsilon}(\mu,\nu) < \mathcal{W}_{c,\epsilon}(\nu,\nu)$.

To circumvent these issues, the \textit{debiased Sinkhorn divergence} was proposed in~ \cite{ramdas2017wasserstein}.
It is given by
    \begin{equation}
    \label{eq:sinkhorndivergence}
        \mathcal{S}_\epsilon(\mu,\nu) := \mathcal{W}_{c,\epsilon}(\mu, \nu) - \frac{1}{2}\left(  \mathcal{W}_{c,\epsilon}(\mu,\mu) +  \mathcal{W}_{c,\epsilon}(\nu,\nu) \right) 
    \end{equation}
and has been adopted in the machine learning community e.g. to obtain generative models~\cite{genevay2018learning}.
As presented in \cite{feydy2020geometric}, the debiased Sinkhorn divergence may metricize the convergence in law and is thus suitable to be applied to the task of measure fitting as we have in mind.
In particular we have the following result due to \cite{feydy2020geometric, feydy2019interpolating}.
\begin{theorem}
\label{theorem:sinkhorn}
Let $\mathcal{V}$ be a compact metric space with a Lipschitz cost function $c\colon\mathcal{V}\times\mathcal{V}\to\mathbb{R}_+$ that induces a positive universal kernel $k_\epsilon(\cdot,\cdot)=\exp(-c(\cdot,\cdot)/\epsilon)$ for $\epsilon > 0$.
Then, $\mathcal{S}_\epsilon$ defines a \textbf{symmetric, positive, definite, smooth} loss function that is \textbf{convex} in each of its input variables.
It also metricizes the convergence in law: for all probability measures $\mu,\nu \in \mathcal{D}(\mathcal{V})$,
\begin{align}
    0 = \mathcal{S}_\epsilon(\nu,\nu) \leq \mathcal{S}_{\epsilon}(\mu, \nu),\\
    \mu = \nu \Leftrightarrow \mathcal{S}_\epsilon(\mu,\nu) = 0 ,\\
    \mu_n\rightharpoonup \mu \Leftrightarrow \mathcal{S}_\epsilon(\mu_n,\mu) \to 0.
\end{align}
The results extend to measures with bounded support on an Euclidian space $\mathcal{V}=\mathbb{R}^d$ endowed with ground cost functions $c(x,y)=\|x-y\|$ or $c(x,y)=\frac{1}{2}\|x-y\|^2$, which induce exponential and Gaussian kernels, respectively.
\end{theorem}
The smoothness of the debiased Sinkhorn divergence with explicit derivation of the gradient is already pointed out in~\cite{luise2018differential} in the discrete measure setting.

The debiased Sinkhorn divergence can be interpreted as an interpolation between the exact Wasserstein metric and kernel maximum mean discrepancies~\cite{ramdas2017wasserstein,genevay2018learning}.
This observation can be seen in the sampling complexity result obtained for the entropic regularized Wasserstein distance $\mathcal{W}_{c,\epsilon}$ due to~\cite{genevay2019sample}.
\begin{theorem}
\label{theorem:entropicSampleComplex}
Let $\mu,\nu$ be a probability on bounded subsets $\mathcal{X},\mathcal{Y}\subset\mathbb{R}^N$ and let $\mathcal{W}_\epsilon$ be defined upon a smooth $L$-Lipschitz cost $c$.
Then, 
\begin{align}
    \mathbb{E}\left[|\mathcal{W}_\epsilon(\mu,\nu) - \mathcal{W}_\epsilon(\mu_n,\nu_n)| \right]
    \in \mathcal{O} \left(
    \frac{\exp(\kappa / \epsilon)}{\sqrt{n}}\left(1 + \frac{1}{\epsilon^{\lceil N/2 \rceil}} \right)\right),
  \end{align}
 with $\kappa = 2|\mathcal{X}|L + \|c\|_\infty$ and the constants only depend on $|\mathcal{X}|,|\mathcal{Y}|, d$ and $\|c^{(k)}\|_\infty$ for $k=0,\ldots,\lceil N/2 \rceil$.
 In particular, we get the following asymptotic behaviour in $\epsilon$,
 \begin{align}
     \mathbb{E}\left[|\mathcal{W}_\epsilon(\mu,\nu) - \mathcal{W}_\epsilon(\mu_n,\nu_n)| \right] &\in \mathcal{O}\left(\frac{\exp(\kappa/\epsilon)}{\epsilon^{\lceil N/2\rceil }\sqrt{n}}\right),\quad \epsilon \to 0, \\
     \mathbb{E}\left[|\mathcal{W}_\epsilon(\mu,\nu) - \mathcal{W}_\epsilon(\mu_n,\nu_n)| \right]
     &\in \mathcal{O}\left(\frac{1}{\sqrt{n}}\right),\quad \epsilon\to +\infty.
 \end{align}
\end{theorem}
For $\epsilon \to\infty$, we obtain a rate of convergence independent on the dimension $N$, while for $\epsilon \to 0$ the curse of dimensionality appears in terms of $\epsilon$ itself.
The question of the error introduced by the entropic regularization to the  Wasserstein distance was answered in \cite{genevay2019sample}. 
\begin{theorem}
\label{theorem:entropicApproximation}
Let $\mu$ and $\nu$ be probability measures on $\mathcal{X},\mathcal{Y}\subset\mathbb{R}^N$ such that $|\mathcal{X}|=|\mathcal{Y}|\leq D>0$ and assume that the cost $c$ is $L$-Lipschitz with respect to both arguments.
Then it holds
\begin{align}
    0\leq \mathcal{W}_\epsilon(\mu,\nu) - \mathcal{W}(\mu,\nu) \leq 2\epsilon N \log\left(\frac{e^2\cdot L\cdot D}{\sqrt{N}\cdot\epsilon}\right).
\end{align}
In particular, as $\epsilon \to 0$,
\begin{minipage}{0.49\linewidth}
\begin{align}
    0\leq \mathcal{W}_\epsilon(\mu,\nu) - \mathcal{W}(\mu,\nu)\sim 2\epsilon N\log(1/\epsilon).
\end{align}
\end{minipage}
\end{theorem}
In~\cite{luise2018differential} an improved approximation for discrete measures is provided if the optimal transport plan is accessible.
Concretely, for discrete measures $\mu,\nu$, let 
\begin{align} \pi^\ast = \operatorname{argmin}\limits_{\pi\in\mathcal{D}(\mathcal{V}\times\mathcal{V})}\langle \pi, c\rangle +   \epsilon
    \operatorname{KL}(\pi,\mu\otimes\nu) \\
    \text{subject to } \quad \pi\geq 0, \quad \pi_1=\mu, \quad \pi_2=\nu.
\end{align}
Then, for some $C>0$, 
$$
| \langle \pi^\ast, c \rangle - \mathcal{W}(\mu,\nu) | < C e^{-1/\epsilon}.
$$
However, the solution of~\eqref{eq:discreteOT} in general lacks sparsity due to the regularization and one hence only obtains ``blurred versions'' of the sparse optimal transport plan of the Wasserstein distance~\cite{peyre2019computational}.

Note that all results above extend to the debiased Sinkhorn divergence itself.
Finally, Theorems~\ref{theorem:entropicSampleComplex} and~\ref{theorem:entropicApproximation} allow to define a ``sweat spot choice'' of $\epsilon$ for the numerical evaluation of $\mathcal{W}_\epsilon$ as discussed in the following section.

\subsection{Computation of discrete OT}
\label{sec:computation discrete OT}

The following primal problem is the discrete counterpart to~\eqref{eq:discreteOT} for the entropic regularization
\begin{equation}
\label{eq:primal}
\mathcal{W}_{c,\epsilon}(\mu_n,\nu_m) = \min\limits_{\pi\in \mathcal{U}_{nm}} \langle \pi, C\rangle + \epsilon\langle \pi,\log\pi-\bm{1}_{n,m}\rangle,
\end{equation}
where $\bm{1}_{n,m}\in\mathbb{R}^{n,m}$ is a matrix with all entries equal to one and component-wise application of the logarithm.
Defining the kernel matrix $K:=\exp(-C/\epsilon)$ with component-wise application of the exponential and recalling the definition of weights $a$ and $b$ as in Section~\ref{sec:exacttransport}, the dual problem reads
\begin{equation}
\label{eq:dual}
    \sup\limits_{f\in\mathbb{R}^n,g\in\mathbb{R}^m} -\left\langle a, f\right\rangle - \left\langle b,g\right\rangle - \epsilon \left\langle \exp( -f/\epsilon), K \exp(-g/\epsilon)\right\rangle.
\end{equation}
The primal and the dual formulation are linked via the optimality conditions
\begin{equation}
\label{eq:optimalityCond}
    \left\{
    \begin{array}{rcl}
       \pi    & =  & \operatorname{diag}(\exp\{-f/\epsilon\}) K \operatorname{diag}(\exp\{ -g/\epsilon\}),\\
       a  & = & \operatorname{diag}(\exp\{-f/\epsilon\})
     K\exp\{-g/\epsilon\}, \\
     b &= &
     \operatorname{diag}(\exp\{-g/\epsilon\})
     K^\intercal \exp\{-f/\epsilon\}.
    \end{array}
    \right.
\end{equation}
These form the starting point for the numerical computation. 
The classical Sinkhorn–Knopp algorithm~\cite{cuturi2013sinkhorn} relies on a fixed point iteration of the function
$$
\Phi(u,v) = \left[
\begin{array}{ll}
b \oslash K^T v\\ a\oslash K u 
\end{array}
\right] \in \mathbb{R}^{n+m},
$$
where $\oslash$ denotes pointwise division of vectors.
This scheme is related to the last two conditions in~\eqref{eq:optimalityCond} with $(u^\ast, v^\ast)$ denoting the fixed point with $u^\ast = \exp(f^\ast/\epsilon)$, $v^\ast = \exp(g^\ast/\epsilon)$ and the unique solution to~\eqref{eq:dual} $f^\ast,g^\ast$.
Let $\delta$ be a required tolerance for the iteration error.
Then the Sinkhorn algorithm outputs an approximated transportation plan in $\mathcal{O}(\delta^{-2}n^2 \|C\|_\infty^2 \ln n )$ iterations~\cite{altschuler2017near,dvurechensky2018computational}. 
Alternatively,~\eqref{eq:primal} or~\eqref{eq:dual} might be solved with other solvers such as an ``Adaptive Primal-Dual Accelerated Gradient Descent''~\cite{dvurechensky2018computational} in $\mathcal{O}\left(\min \left\{
\delta^{-1}n^{9/4}\sqrt{\|C\|_\infty \ln n },
\delta^{-2}n^2 R\|C\|_\infty \ln n
\right\}
\right)$ or with a Newton scheme~\cite{brauer2017sinkhorn} which results in local quadratic convergence.
Since the regularization parameter $\epsilon$ causes numerical instabilities in the iterative process as $\epsilon\to 0$, a stabilized Sinkhorn algorithm variant has been formulated in $\log$ space by several authors~\cite{dvurechensky2018computational, schmitzer2019stabilized}.

With a large number of point data $n,m\gg 1$, these classic formulations may still become computationally cumbersome or infeasible.
Hence, modern state-of-the-art computations of regularized optimal transport or Sinkhorn divergences rely on $\epsilon$-scaling heuristics, adaptive kernel truncation or multiscale schemes~\cite{schmitzer2019stabilized}, \textit{e.g.} enhanced by $k$-mean clustering~\cite{feydy2020geometric}.
A very promising recent approach relies on streamed sample data in a so-called online Sinkhorn scheme~\cite{mensch2020online} to better approximate the original continuous regularized optimal transport problem with a sampling complexity arbitrarily close to $\mathcal{O}(1/N)$. 
From the vast amount of codes for computational transport, we only refer to \verb!GeomLoss! by Jean Feydy~
\cite{feydy2019interpolating} with embarrassingly parallel GPU implementations and a reference online implementation with linear memory footprint.
The present work relies on this library for the numerical experiments, which allows to carry out the measure fitting with efficient computation of the gradients with sample data $n>\num[round-precision=2,round-mode=figures]{1e+06}$ in seconds.

\section{Model class and stochastic reference coordinate system}
\label{sec:modelclass}
The following section is devoted to the design of the model class $\mathcal{M}$ and the underlying stochastic reference coordinate system $X$, defining the parameter dependent random variable $\mathcal{M}(X)$ that we aim to fit closely to $Y$ in distribution.
The first section is devoted to the actual dimension of the random variable $X$.
We introduce our model class based on compressed multi-element generalized polynomial chaos. Finally, we formulate the optimization problem and discuss involved challenges.

\subsection{Dimensionality}
\label{sec:dimensionality}

It is inevitable to choose a suitable reference coordinate system for an efficient approximation of the unknown $Y$ in distribution.
Given a sufficient amount of samples of $Y$ such that empirical marginals can be inferred with some confidence, the question is if we can make an informed guess about the dimensions of the reference coordinate system $X$.
It turns out that at first glance, the choice is very flexible as stated next.

\begin{theorem}
Let $X, Y$ be continuous random vectors with images  $\operatorname{img}X\subset \mathbb{R}^M$ and $\operatorname{img}Y\subset\mathbb{R}^N$ for $1\leq N,M < \infty$ such that there exists an invertible transformation $T_Y$ with uniformly distributed $T_Y(Y)$ on $[0,1]^M$.
Then there exists $f\colon \mathbb{R}^M\to \mathbb{R}^N$ such that $Y \equald f(X)$.
\end{theorem}
\begin{proof}
Consider transformations $T_X$ and $T_Y$ such that $T_X(X)$ and $T_Y(Y)$ are uniformly distributed over $[0,1]^M$ and $[0,1]^N$, respectively.
Those transformations may be constructed via the marginal cumulative distribution functions as in the Rosenblatt transformation~\cite{lebrun2009rosenblatt}.
Since $[0,1]^M$ and $[0,1]^N$ have same cardinality, there exists a bijection $\Phi\colon[0,1]^M\to[0,1]^N$.
Define the random variable $Z:=\Phi( T_X(X))$ with image $\operatorname{img}(Z)\subset [0,1]^N$.
Furthermore, let $T_Z$ be a transformation such that $T_Z(Z)$ is uniformly distributed on $[0,1]^N$.
Then by construction, 
\begin{equation}
\label{eq:equaldistributionYfX}
     Y \equald f(X), \quad f:= T_Y^{-1} \circ T_Z \circ \Phi \circ T_X. \qedhere
\end{equation}
\end{proof}
We hence found a representation of a continuous $N$-dimensional random vector based on an $M$-dimensional stochastic reference coordinate system $X$.
Note that $f$ in general may not be bijective, highly non-linear and not necessarily continuously differentiable.
Moreover, the choice of $f$ is not unique.

In practice, we may look for the components of $Y$ that are highly correlated and choose the dimension of $M$ dependent on $N$ and the number of highly correlated components.
This is motivated by the fact that the image of highly correlated components might be close to a lower-dimensional structure and thus could be easier approximated in distribution in a lower-dimensional reference coordinate system. 

As the other extreme, $Y$ may consist of fully uncorrelated components but have a simple one-dimensional dependency to be explored.
This setting is discussed in the following example.

\begin{example}
\label{ex:uncorrelated}
Let $X\sim\mathcal{U}(-\pi,\pi)$.
Define $Y:=(\sin(X), \cos(X), \sin(2X),\ldots, \cos(nX))$ with images in $\mathbb{R}^{2n}$ for $n\in\mathbb{N}$.
Then all components of $Y$ are uncorrelated, in particular $\mathbb{E}[Y_i,Y_j]= c_{ij}\partial_{ij}$ for some constant $c_{ij}\in\mathbb{R}$.
\end{example}

\subsection{Polynomial chaos expansion and STT compression}
\label{sec:pcett}

Let $X\sim \mu = \mu_0 \otimes \ldots \mu_0$ be an $M$-dimensional finite reference coordinate system with $X=(X_1,\ldots, X_M)$ and independent reference random variables $X_i \sim \mu_0$, $i=1,\ldots,M$ such that $\mu_0$ has finite arbitrary moments.
Let $\{p_{i}, i\in\mathbb{N}_0\}$ with $p_0 \equiv 1$ be a set of orthogonal polynomials w.r.t. to $\mu_0$, i.e. $\mathbb{E}[p_i(X_1)p_j(X_1)] := \partial_{ij}$. 
We define the set of tensorized stochastic polynomials 
\begin{equation}
    \{P_\alpha, \alpha\in \mathbb{N}_0^M\},\quad P_{\alpha}(X) = \prod\limits_{i=1}^M p_{\alpha_i}(X_i).   
\end{equation}
These polynomials are orthogonal w.r.t. $\mu$, in particular $\mathbb{E}[P_\alpha(X)P_\beta(X)]=\partial_{\alpha\beta}$ for $\alpha,\beta\in\mathbb{N}_0^M$.
Such type of polynomial construction typically arise in the context of stochastic Galerkin schemes with convergence in $L^p(\Omega,\mu, \mathcal{V}))$  for some separable Banach space $\mathcal{V}$ \cite{ghanem2003stochastic, xiu2002wiener, schwab2011sparse, ernst2012convergence,dunkl2014orthogonal} even for the case $M=\infty$.
Given an unknown random variable $Y$ taking values in $\mathbb{R}^N$, we look for an approximation in distribution of the form 
\begin{align}
\label{eq:PCEapprox}
    Y  \apprxd Y_{\mathrm{PCE}}:=  \sum\limits_{\alpha\in\mathbb{N}_0^M} C[\alpha] P_\alpha(X),
    \quad C\colon \mathbb{N}_0^M\to \mathbb{R}^N.
\end{align}
Practically, we aim for a truncated version of the polynomial chaos expansion.
Consider a finite index set $\Lambda\subset\mathbb{N}_0^M$, $0\in\Lambda$ with $0\in\mathbb{N}_0^M$ and $|\Lambda|<\infty$.
We define the truncated polynomial chaos expansion representation $Y_\Lambda\apprxd Y_{\mathrm{PCE}}$ by 
\begin{align}
\label{eq:truncatedPCE}
     Y_\Lambda := \sum\limits_{\alpha\in \Lambda} C_\Lambda[\alpha] P_\alpha(X),\quad C_\Lambda \colon \Lambda \to \mathbb{R}^N.
\end{align}
Representation~\eqref{eq:truncatedPCE} may suffer from the curse of dimensionality since the set $\Lambda$ grows exponential in $M$ without additional sparsity assumptions.
In order to see this, consider the  tensorized index set
\begin{equation}
    \Lambda = \bigtimes\limits_{m=1}^M\Lambda_m, \quad\Lambda_i\subset\mathbb{N}_0.
\end{equation}
If $|\Lambda_m|=d\in\mathbb{N}$ for $m=1,\ldots M$, then the tensor $C_\Lambda$ has $d^M$ inputs, resulting in a total of $N d^M$ degrees of freedom in this approximation class.
A possible technique to circumvent such exponential complexity can be found with tensor compressions by so-called \textit{low-rank formats}.
Popular tensor formats in the literature are tensor trains (TT)~\cite{oseledets2011tensor} or hierarchical tensors (HT)~\cite{hackbusch2012tensor} for $N=1$.

In the numerical investigations we observe an unfavourable exponential growth of ranks with growing $N$.
To alleviate this, we consider a \textit{tensor network format} as our compression tool, namely stacks of tensor trains (STT).
To make this concrete, define $r_0=r_M=1$ and let $r_m\in\mathbb{N}$ for $m=1,\ldots, M-1$ and $d_m=|\Lambda_m|$.
Moreover, choose $C_\Lambda$ of the form
\begin{align}
\label{eq:tensorringdecomp}
    C_\Lambda[\alpha] = C_{\Lambda_1}[\alpha_1]\bullet \ldots \bullet C_{\Lambda_M}[\alpha_M],\quad
    C_{\Lambda_j}[\alpha_j] \colon \mathbb{N}^{r_{j-1}, r_j} \to \mathbb{R}^N.
\end{align}
Here, $\bullet$ denotes the contraction between adjacent indices and the Hadamard product w.r.t. to the output dimension $N$.
Specifically, for $A\colon \mathbb{N}^{q,r}\to\mathbb{R}^N$, $B\colon \mathbb{N}^{r,s}\to\mathbb{R}^N$ and with Hadamard product $\odot$,
\begin{equation}
\label{eq:constraction}
    A \bullet B = \sum\limits_{k=1}^{r} A[:,k] \odot B[k,:] \colon \mathbb{N}^{q,s}\to\mathbb{R}^N.
\end{equation}
We may identify the output dimension $N$ with an additional index $i$ as a tensor representation.
With some abuse of notation, we let $C_\Lambda[i, \alpha] = (C_\Lambda[\alpha])_i$ such that~\eqref{eq:tensorringdecomp} can be rewritten as
\begin{align}
\label{eq:tensorringpointwise}
    C_\Lambda[i, \alpha] = C_{\Lambda_1}[i, \alpha_1]\cdot\ldots\cdot C_{\Lambda_M}[i, \alpha_M] \in\mathbb{R},\qquad i=1,\ldots,N.
\end{align}
Thus, the components $C_{\Lambda_1}$ define tensors of order $4$ and we compactly use the notation and ordering of indices as  
\begin{equation}
\label{eq:coeffdecomp}
    C_{\Lambda_m} =  \left(C_{\Lambda_m}[k_{m-1}, i, \alpha_m, k_m]\right)\in \mathbb{R}^{r_{m-1}, N, d_m, r_m},\quad m=1,\ldots,M
\end{equation}
where $i=1,\ldots,N$ resorts to the second and $\alpha_m\in \Lambda_m$, $|\Lambda_m|=d_m$ resorts to the third index, respectively.
In summary, we seek a stack of tensor trains with uniform rank $\bm{r}=(r_0,\ldots, r_M)$ with the convention $r_0=r_M=1$ of the form
\begin{align}
    \label{eq:fulltenorringrespresentation}
    Y_\Lambda =  \sum\limits_{k_1=1}^{r_1}\ldots\sum\limits_{k_{m-1}=1}^{r_{m-1}}
                \bigodot\limits_{m=1}^M C_{\Lambda_m}[k_{m-1}, :, \alpha_m, k_m] p_{\alpha_m}(X_m),
\end{align}
where the colon ``$:$'' denotes the vector representation of output dimension $N$.
Note that while the representation in~\eqref{eq:fulltenorringrespresentation} looks complicated at first glance, it is only a contraction of tensors of order $4$, which can be realized efficiently in practice using the Einstein summation convention.

\begin{remark}
    The choice of uniform rank $\bm{r}$, i.e. the same TT rank for each component $i=1,\ldots, N$, in this work is only for simplicity of presentation.
    It is straightforward to extend to the case of individual TT ranks within the stack, which would potentially lead to a more efficient representation.
\end{remark}

Once such a representation~\eqref{eq:fulltenorringrespresentation} is available, we can conveniently compute statistics such as moments in a sampling-free manner, which is a crucial benefit of such formats.
This is due to the separated structure in the dependence of the input reference coordinate system $X=(X_1,\ldots, X_m)$.
The required high-dimensional quadrature then corresponds to simple index contractions.

Assume $r_m =r\in\mathbb{N}$ for $m=1,\ldots, M-1$ and $d_m=d\in\mathbb{N}$ for $m=1,\ldots,M$.
Then, the degrees of freedom of the proposed STT format are given by
\begin{align*}
   N\left( \sum_{m=1}^M r_{m-1}d_m r_{m} - \sum_{m=1}^{M-1} r_m^2\right) = N ( (2r + (M-2)r^2 )d - (M-1)r^2) \in \mathcal{O}(NMd).
\end{align*}
This means that a linear dependence on the input dimension $M$ related to approximating $X$ can be achieved, effectively circumventing the curse of dimensionality that would otherwise occur with a full tensor with $N M^d$ entries.
We point out that the STT format of fixed rank is closed with respect to the Frobenius norm. This property is inherited directly from the closedness of the classical TT format.

For a more specific ``dependence pattern'' of $Y_\Lambda$, we can set entries in the components $C_{\Lambda_m}$ denoted as \textit{degrees of freedom} to zero and thus specify the dependence of the output.
As an example, for given $i\in\{1,\ldots,N\}$, setting
$$C_{\Lambda_m}[k_{m-1},i,\alpha_m,k_m]=0,\quad k_{m-1}=1,\ldots,r_{m-1}, \alpha_m =1,\ldots, d_m, k_m = 1,\ldots,r_m,$$
corresponds to no dependence of $(Y_\Lambda)_i$ on $X_m$.
In addition to the practical meaning of the model, this also reduces the number of degrees of freedom in the optimization process.
Special cases involve that each component in $Y_\Lambda$ is modeled independently with a univariate polynomial chaos expansion or a triangular parameter dependence structure as in Knothe-Rosenblatt transport maps in the case of $N=M$, see~\cite{zech2021sparse} for a recent result.

\subsubsection{Multi-element polynomial chaos}
\label{sec:multielementpce}
A straight-forward generalization of the polynomial chaos as model class is the multi-element generalized polynomial chaos (GPC)~\cite{wan2006multi} described in this section.
Let $X=(X_1,\ldots,X_M)$ denote a reference coordinate system with image $\mathcal{X}=\operatorname{img}(X)\subset\mathbb{R}^M$ and density $f_X$.
Moreover, let $\mathcal{X}^s$ denote a disjoint partition of $\mathcal{X}$ for $s=1,\ldots, S\in\mathbb{R}$.
Assume that the underlying distribution $\mu$ of $X$ has arbitrary finite moments, then one can define a set of functions $\{P_\alpha^s:\alpha \in \mathbb{N}_0^M\}$ that are piecewise polynomials on $\{\mathcal{X}^s\}$ with support on $\mathcal{X}^s$ and also orthonormal with respect to the expectation.
In particular, it holds for $\alpha,\beta\in\mathbb{N}_0^M, s,s'\in\mathbb{N}$ that
$$
\mathbb{E}[P_\alpha^s(X)P_\alpha^{s'}(X)] = 
\int\limits_{\mathcal{X}^s\cup\mathcal{X}^{s'}} \chi_{\mathcal{X}^s\cap \mathcal{X}^{s'}}(x)
P_\alpha^s(x)P_\beta^{s'}(x) f_X(x)\mathrm{d}\lambda(x)
= 
\delta_{\alpha,\beta}\delta_{s,s'}, 
$$
where $\chi$ denotes the indicator function.

On the one hand, due to the locality of the function, with this notion a model class is able to move masses of probability separately subject to the partition, which is illustrated in Figure~\ref{fig:multimodalsketch}.
On the other hand, if $\mathcal{X}$ is compact and $Y$ is a multimodal random variable such that each accumulation of mass is completely disjoint, there exists no continuous model class $\mathcal{M}$ such that $Y\equald \mathcal{M}(X)$ since probability mass would have to be split.
This in turn motivates the usage of piecewise continuous approximation classes.

\begin{remark}
In practice, to identify multimodal behaviour, one may utilize methods based on $k$-means clustering to determine the number of modes $k$, \textit{e.g.} using the G-means algorithm of~\cite{hamerly2004learning}.
This can then be used to define an adequate partitioning of $S$.
Note that the definition of the partition and thus the structure of non-continuity can be chosen arbitrarily as long as the resulting generation of reference coordinate samples is well balanced across the subdomains.
One should specifically avoid situations where only few samples are in some $\mathcal{X}^s$, resulting in limited available information for fitting measures.
\end{remark}

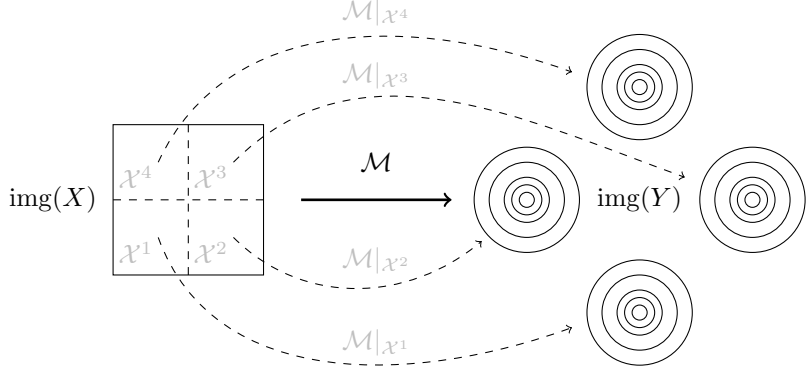
\begin{figure}
\begin{center}
\begin{tikzpicture}

\draw (-7,-1) -- ( -5, -1) -- ( -5 ,1) -- ( -7, 1) --cycle;

\draw[dashed] (-6,-1) -- (-6,1);
\draw[dashed] (-7, 0) -- (-5,0);

\foreach \x in {0.1, 0.2, 0.3, 0.5, 0.7}{
    \draw (1.5,0) circle (\x);
    \draw (0,1.5) circle (\x);
    \draw (0,-1.5) circle (\x);
    \draw (-1.5, 0) circle (\x);
}

\draw[->, line width = 1] (-4.5, 0) -- (-2.5,0);

\draw[->, dashed] (-5.4,0.5)  to [out=45,in=160] (0.6,0.3);

\draw[->, dashed] (-6.4, 0.5) to [out=65, in = 160] (-0.9, 1.7);

\draw[->, dashed] (-6.4, -0.5) to [out=-70, in=-160] (-0.9, -1.7);

\draw[->, dashed] (-5.4, -0.5) to [out=-45, in=-140] (-2.1,-0.55);

\node at (-3.5,0.5) {$\mathcal{M}$};
\node at (-7.8, 0) {$\operatorname{img}(X)$};
\node at (0, 0) {$\operatorname{img}(Y)$};

\node[opacity = 0.5, gray] at (-5.7, 0.3) {$\mathcal{X}^3$};
\node[opacity = 0.5, gray] at (-6.7, 0.3) {$\mathcal{X}^4$};
\node[opacity = 0.5, gray] at (-5.7, -0.7) {$\mathcal{X}^2$};
\node[opacity = 0.5, gray] at (-6.7, -0.7) {$\mathcal{X}^1$};

\node[opacity=0.5, gray] at (-3.5, -1.9) {$\mathcal{M}|_{\mathcal{X}^1}$};

\node[opacity=0.5, gray] at (-3.5, -0.8) {$\mathcal{M}|_{\mathcal{X}^2}$};

\node[opacity=0.5, gray] at (-3.5, 1.65) {$\mathcal{M}|_{\mathcal{X}^3}$};

\node[opacity=0.5, gray] at (-3.5, 2.5) {$\mathcal{M}|_{\mathcal{X}^4}$};

\end{tikzpicture}
\caption{Sketch of a separated propagation of masses through a model class based on piecewise continuous representations, \textit{e.g.} realized via multi-element polynomial chaos.}
\label{fig:multimodalsketch}
\end{center}
\end{figure}

\subsection{The optimization problem}
\label{sec:optimization problem}

In the following the optimization procedure is described that allows to obtain a transport representation as described before.
Let $\theta$ denote the degrees of freedom of the parametric model class $\mathcal{M}(\cdot) = \mathcal{M}[\theta](\cdot)$, which \textit{e.g.} could be a STT format with underlying (multi-element) polynomial chaos as feature functions.
Here, $\theta$ relates to the sub-parts of the STT decomposition of the coefficient tensor as in~\eqref{eq:coeffdecomp}.
For a given stochastic coordinate system $X$ and a batch of realizations $\bm{Y}\in\mathbb{R}^{n, N}$ of the random variable $Y$, compute $m$ samples of $X$ represented as a batch of samples $\bm{X}\in \mathbb{R}^{m, M}$.
Then, for fixed $\theta$, $\mathcal{M}[\theta][\bm{X}]\in \mathbb{R}^{m, N}$ defines a batch of samples that can be compared to $\bm{Y}$ in terms of empirical measures $\mu_n$ and $\nu_m=\nu_m[\theta]$ and the debiased Sinkhorn divergence.
The respective optimization problem reads
\begin{equation}
    \label{eq:optprob}
    \min\limits_{\theta} \mathcal{S}_\epsilon(\mu_n,\nu_m[\theta]).
\end{equation}
Based on the dependence structure of $\mathcal{M}$ on $\theta$, \eqref{eq:optprob} defines a non-convex non-linear optimization problem with an almost everywhere differentiable function.
Since the STT format of fixed rank forms a Riemannian manifold and the loss $S_{\epsilon}$ is smooth, we can utilize Riemannian descent methods for hierarchical low-rank formats based on automatic differentiation~\cite{GM2024} to obtain a local minimum of \eqref{eq:optprob}.

We demonstrate the non-convexity and non-linearity for a very simple model class with linear parameter dependence in Example~\ref{example:locMinima}.

\begin{example}
\label{example:locMinima}
Let $X\sim\mathcal{U}(-1,1)$ be uniformly distributed and denote by $L_k$ the $k$-th Legendre polynomial defined on $[-1,1]$.
Moreover, let $Y = L_2(X)$ and the model class $\mathcal{M}(x) = c_0 + c_2L_2(x)$.
Let $\theta = [c_0,c_2]$, then $\theta^\ast = [0,-1]$ and $\theta^\ast=[0,1]$ define local minima of the function $\theta\mapsto \mathcal{S}_\epsilon(\mu_n,\nu_m[\theta])$ up to statistical noise introduced by the finite number of samples $n,m\in\mathbb{N}$.
The left plot in Figure~\ref{fig:localminiCuttoff} illustrates this example in log-scale.
\end{example}

\begin{proposition}
\label{prop:oddfunction}
Let $X_1,X_2\colon \Omega \to \mathbb{R}^M$ with $X_1\equald -X_1 \equald X_2$ and $f\colon\mathbb{R}^M\to\mathbb{R}^N$ be measurable.
Then $f(X_1) = f(X_2)$ in distribution.
Moreover, if $f$ is odd, it holds that $f(X_1) = -f(X_2)$ in distribution.
\end{proposition}
\begin{proof}
Since $f$ is odd it holds that $f(-x)=-f(x)$ for $x\in \Xi$.
Given that $X_1=X_2$ in distribution and $X_1 = -X_1$ in distribution 
it follows $f(X_1) = f(-X_1) = -f(X_1) = -f(X_2)$.
\end{proof}
As a consequence, given an orthogonal model class with respect to the  reference coordinate system $X\sim \mathcal{D}(\mathbb{R}^n)$ of a symmetric distribution $\mathcal{D}(\mathbb{R}^n)$, we can use a decomposition like
\begin{align}
    \mathcal{M}(x) =  \sum\limits_{\alpha\in\mathrm{ODD}} C[\alpha] P_\alpha(x)+ \sum\limits_{\alpha\notin\mathrm{ODD}}C[\alpha] P_\alpha(x).
\end{align}
Here $\mathrm{ODD}\subset\Lambda$ denotes the set of indices corresponding to odd polynomials.
Due to Proposition~\ref{prop:oddfunction}, we can choose $C[\alpha]\geq 0$ for $\alpha\in\mathrm{ODD}$.
This may reduce the number of local minima significantly, in particular for $d$ getting large.
Typical examples for symmetric distributions are the uniform distribution on $(-1,1)^d$ and the standard normal distribution on $\mathbb{R}^d$.

\begin{example}
\label{example:manylocMinima}
Let again $X\sim\mathcal{U}(-1,1)$ and $L_k$ denote the $k$-th Legendre polynomial defined on $[-1,1]$.
Assume $Y = L_1(X) + L_2(X)$ and the model class $\mathcal{M}$ be given as $\mathcal{M}(x) = c_1L_1(x) + c_2L_2(x)$.
The right plot in Figure~\ref{fig:localminiCuttoff} illustrates the example with multiple local minima in log-scale, including the mentioned reduction effect.
\end{example}

\begin{proposition}
\label{prop:mean}
Let $\mathcal{M}$ be a model class as in Section~\ref{sec:pcett} with orthonormal basis of increasing degree encoded in $\alpha\in\mathbb{N}_0^d$ w.r.t. the underlying stochastic coordinate system $X$.
Then, $C[:, 0] = \mathbb{E}[ \mathcal{M}(X) ]$.
\end{proposition}
\begin{proof}
The proof follows by construction, since the basis is orthonormal w.r.t. the associated measure $\mu$ of $X$. 
\end{proof}
Consequently, given $N$ samples of the unknown random variable $Y$, we can compute the empirical mean and bound $C[:,0]$ around the empirical mean value in the optimization routine.
In particular, in the numerical experiments the obtained optimized parameter always results in the model class matching the empirical mean exactly.
This observation is in agreement with the best possible sample rate of $1/\sqrt{n}$ obtained for large $\epsilon$ since it matches the convergence rate of the empirical mean to the exact mean.

\begin{figure}
    \centering
    \includegraphics[width = 0.49\linewidth]{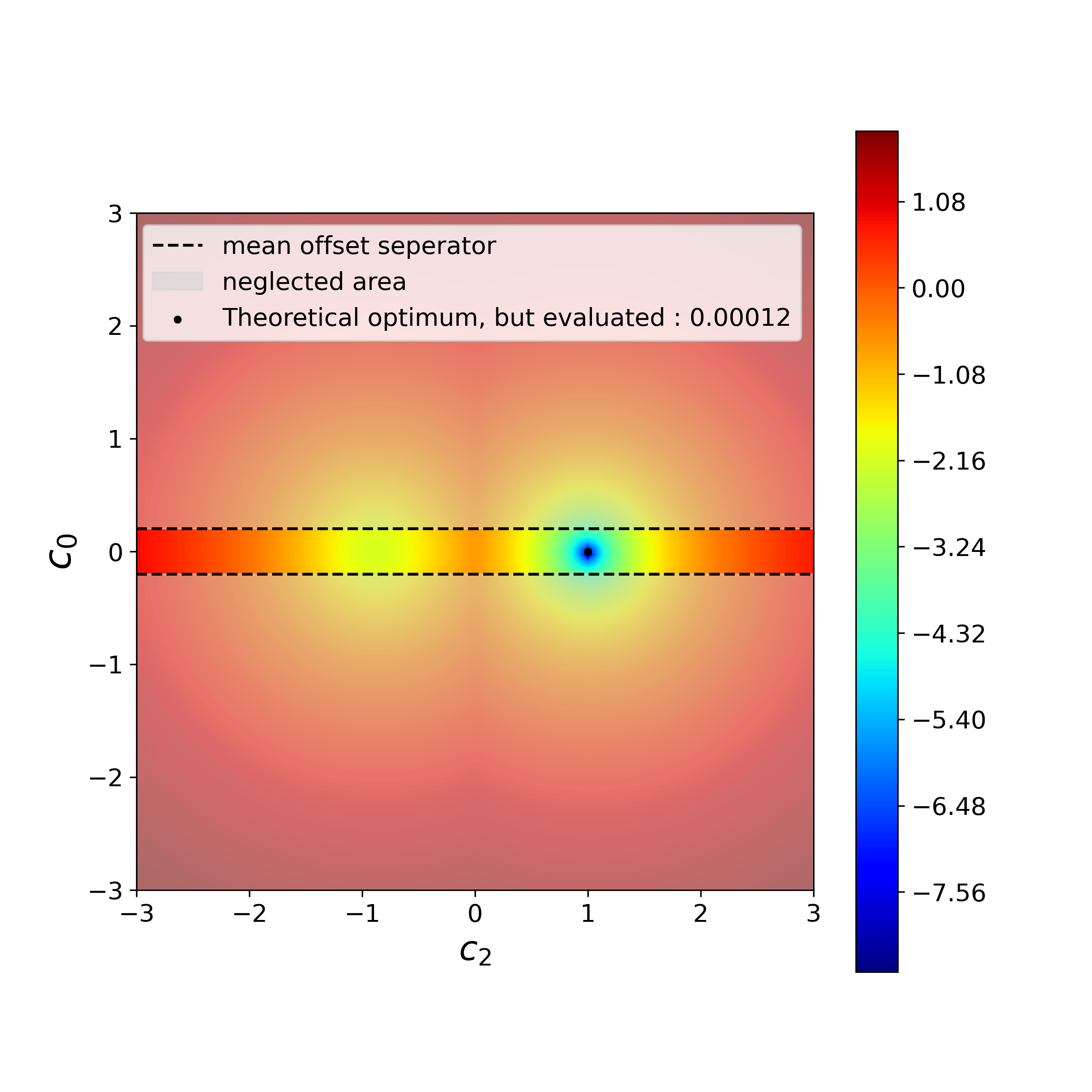}
    \includegraphics[width = 0.49\linewidth]{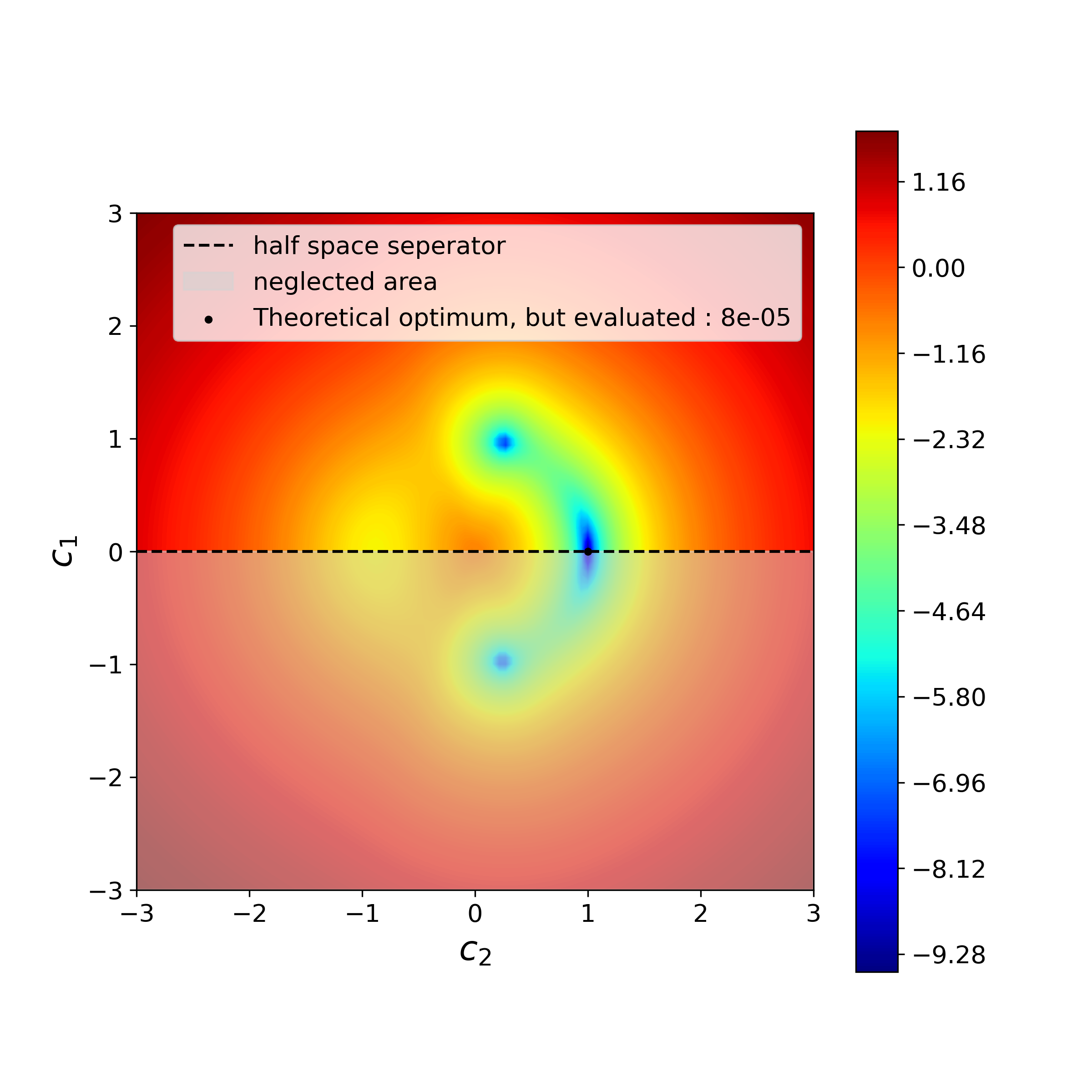}
    \caption{Plots of $\theta\mapsto\mathcal{S}_\epsilon(\mu_n,\nu_m[\theta])$ in log-scale for Example \ref{example:locMinima} (left) and Example \ref{example:manylocMinima} (right) showing the presence of multiple local minima and the effect of neglecting areas in the parameter domain motivated by Proposition \ref{prop:oddfunction} and \ref{prop:mean}.
    }
    \label{fig:localminiCuttoff}
\end{figure}

\newpage
\section{Numerical experiments}
\label{sec:experiments}

In this section we illustrate the performance of the approach presented above with some high-dimensional applications.
 Model classes based on polynomial chaos, its multi-element extension and a compression with the STT format are used.

First, we examine the multimodal random variables where we point out the necessity of the non-continuous model class to represent isolated accumulation of probability masses. 
Second, a simple random field discretization is carried out as a toy example to represent random fields with non-linear but smooth dependence of the stochastic input variable. 
Finally, from a non-periodic random micro-material, an effective upscaled piecewise constant random macro-material is constructed as porosity coefficient of a diffusion problem.

\begin{algorithm}
\caption{Fitting of $Y\apprxd \mathcal{M}(X)$}\label{alg:sinkhornopt}
\begin{algorithmic}
\renewcommand{\algorithmicrequire}{\textbf{Input:}}
\renewcommand{\algorithmicensure}{\textbf{Output:}}
\Require 
$\left\{
\begin{array}{lcl}
 \mu_n \text{ or }\left(Y^i\right)_{i=1}^n\subset \mathbb{R}^N, N\in\mathbb{N}, & \phantom{==} & \triangleright\textit{ data samples} \\
 X\sim \mathcal{D}(\mathbb{R}^M), M\in\mathbb{N},  &        & \triangleright\textit{ input stochastic coordinate system} \\
 m\in\mathbb{N},& & \triangleright\textit{ number of samples drawn from } X, \\
 \mathcal{M}(X)=\mathcal{M}[\theta](X) \text{ based on } \{P_\alpha^s\},,     &      & \triangleright\textit{ model class}    \\
 (\mu,\nu)\mapsto \mathcal{S}_{\epsilon}(\mu,\nu)\in\mathbb{R}_+,& & \triangleright\textit{ Sinkhorn loss operator} \\
 (\theta_0, \ell)\mapsto \theta^\ast = \mathcal{O}(\theta_0,\ell), &  & \triangleright\textit{ local optimizer} \\
\end{array}
\right.$
\Ensure Trained model class with parameter $\theta^\ast$ such that $\mathcal{M}[\theta^\ast](X)\apprxd Y$.\\
\State  Compute base samples $\left(X^j\right)_{j=1}^m\subset\mathbb{R}^M$.
\State Precompute evaluations $P_\alpha^s(X)$ for $s=1,\ldots,S, \alpha \in \Lambda$.
\State  Define $\theta \mapsto \mathcal{M}[\theta](X)$ which defines $\theta\mapsto \nu_m[\theta]$.
\State Set up the loss functional $\theta \mapsto \ell(\theta) := \mathcal{S}_\epsilon(\mu_n,\nu_m[\theta])$.
\State Minimize $\ell(\theta)$ using restarted BFGS and obtain $\theta^\ast$. 
\end{algorithmic}
\end{algorithm}

For the numerical realization, several open source software packages are used.
The python package \verb!geomloss! \cite{feydy2019interpolating} provides the calculation of the debiased Sinkhorn loss and gradient information.
The minimization with automatic differentiation is carried out with \verb!pytorch-minimize!.
The compressed model class is implemented with our python package \verb!TensorTrain!~\cite{tensortrain} while multi-element polynomial chaos and sample generation is realized with \verb!ALEA!~\cite{alea}.
For the numerical upscaling experiment, the composite structures are generated with our library \verb!Bubbles!~\cite{bubbles} and the numerical solution process of the corresponding partial differential equations is realized using \verb!FEniCS!~\cite{fenics}.
Finally, \verb!seaborn!~\cite{seaborn} is used throughout this work for generating plots of the obtained numerical results.

\subsection{Multimodal distribution}
Multimodal distributions especially pose a challenge when representing random variables explicitly.
In this section, we consider multimodal random variables $Y$ modeled as a mixture of Gaussian distributions in $N=1$ and $N=2$ dimensions.
We consider two cases, in the following denoted as \textit{weakly} or \textit{strongly separated}.
By weakly separated we mean multimodal probability masses such that the mass in the transition area between the modes is not negligible, see Figure~\ref{fig:weakmultimodality} for an illustration.
Strongly separated then refers to isolated accumulations of multimodal distributions, see the top right box in Figure~\ref{fig:disjointMultimodalNumeric4Bumps} for an example.

\subsubsection{Weakly separated multimodality}
In the first validation experiment, we consider the case of a weak separated multimodal behavior, i.e. the modes are not fully separated.
Let $T$ be a any diffeomorphic transport map with Lipschitz constant $L>0$ such that $Y = T(X)$ for some random variable $X$ with connected compact image.
Note that as modes separate further and the connection becomes ``thinner'', the Lipschitz constant of the map grows larger unboundedly. 

We consider a one-dimensional random variable $Y$ as a mixture of the Gaussian distributions $\mathcal{N}(-2, 1)$ and $\mathcal{N}(1,0.5)$, \textit{i.e.} $N=1$.
Assume the reference coordinate system $X\sim\mathcal{U}\left([-1,1]^2\right)$ and recall the STT format from Section~\ref{sec:pcett} with global Legendre polynomials of various degree as model class.
In this setting, $M=2$ and the model class does not coincide with (an approximation of) a classical transport map.
We motivate the choice of $M>N$ to allow for a compressed representation of the functional representation $\mathcal{M}(X)$.
In particular in the computation with $M=N$, a polynomial degree of $81$ is required to obtain similar results as in the right plot of Figure~\ref{fig:weakmultimodality}.
This has to be compared to the $M(d+1)r$ degrees of freedom in the compressed format, where $d$ denotes the polynomial degree and $r$ denotes the rank of the STT. 

Figure~\ref{fig:weakmultimodality} illustrates the resulting fitted random variable with respect to the debiased Sinkhorn loss $\mathcal{S}_\epsilon$ with $\epsilon = 0.05$.
We note that the model class is challenged by accurately representing the ``valley'' between the modes.
It can be observed that the approximation quality increases only slightly in this area as the polynomial degree increases.
\begin{figure}
    \centering
    \includegraphics[width = 0.49\linewidth]{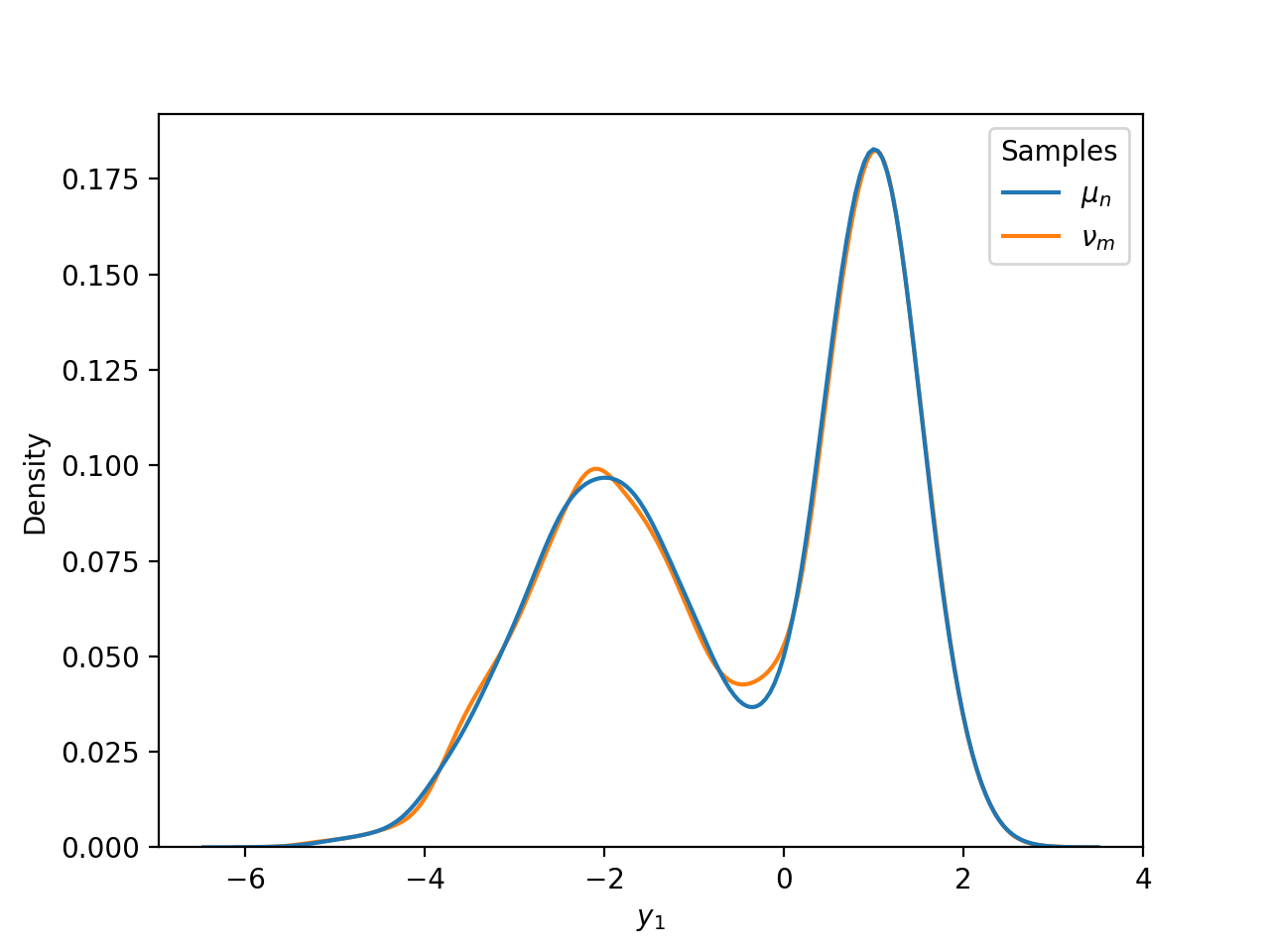}
    \includegraphics[width = 0.49\linewidth]{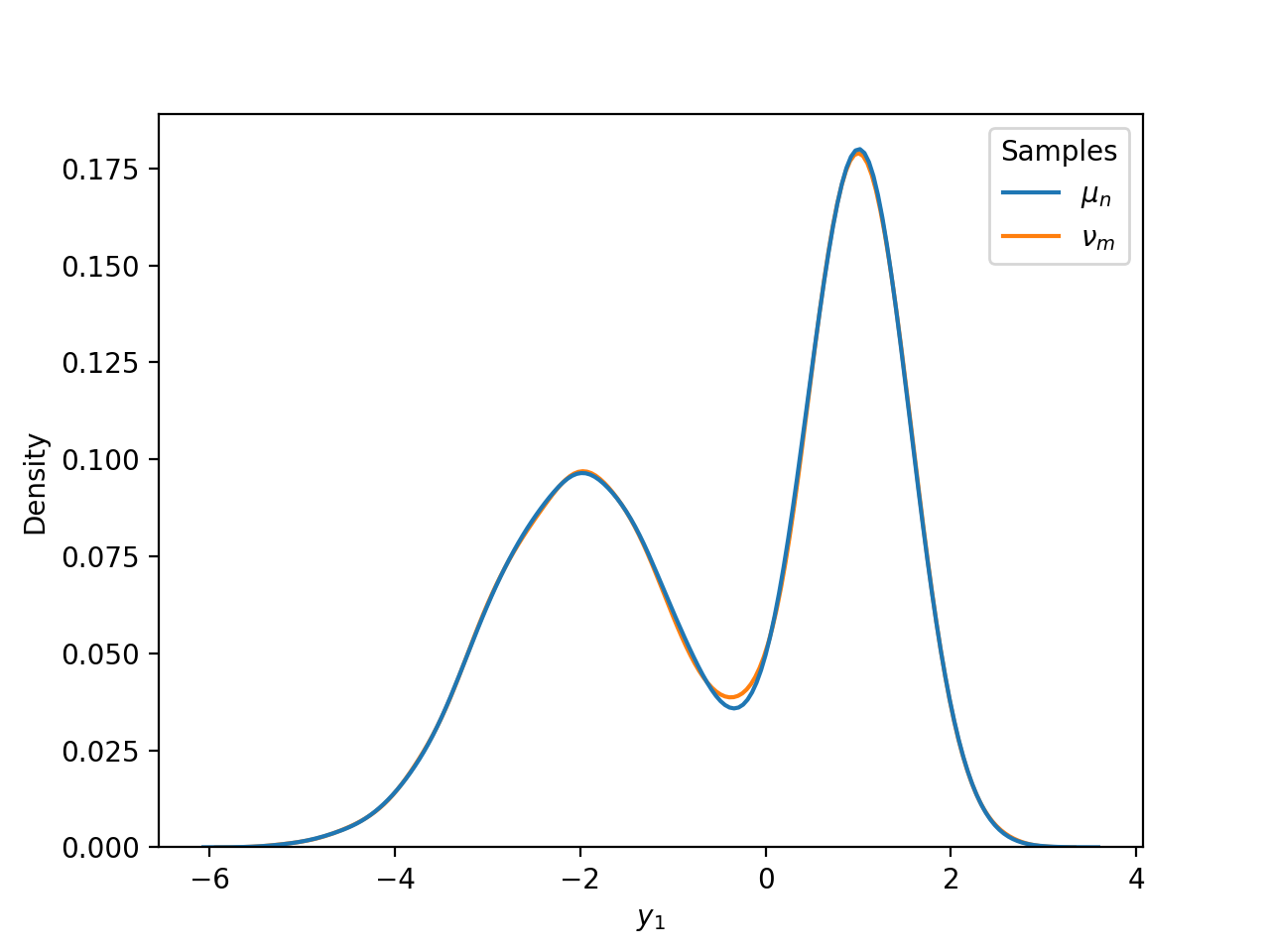}
    \caption{STT fitting of a multimodal distribution. Reference coordinate system $X=(U_1,U_2)$ with $U_1,U_2\sim \mathcal{U}(-1,1)$ \textit{iid} using tensorized Legendre chaos with STT rank $r=(1,3,1)$ based on $K=\num{1e+05}$ samples.
    \textbf{Left:} Polynomial degree $4$ with a minimized debiased Sinkhorn loss of \num{2.0e-04}.
    \textbf{Right:} Polynomial degree $8$ with a minimized debiased Sinkhorn loss of \num{3.6e-05}.
    }
    \label{fig:weakmultimodality}
\end{figure}

\subsubsection{Strongly separated multimodality}
When considering multimodal densities, for disjoint multimodalities with a stochastic input coordinate system $X$ and connected image, the ``transport map'' inevitably exhibits discontinuities to split the separated probability mass as mentioned in Section~\ref{sec:multielementpce}

As a model problem for $Y$ we again consider a mixture of Gaussian distributions.
For a number $B\in\mathbb{N}$ of modes, let
$$
 \Theta[B] = \left\{ 0,\ldots, \frac{2\pi B}{B-1} \right\}.
$$
Then $Y$ is defined as mixture of $$\mathcal{N}(\mu_{S,\theta}, \sigma_B^2I_2), \quad \mu_{S,\theta} = (S\cos(\theta),S\sin(\theta))^T, \quad \theta\in \Theta[B],$$
with shift parameter $S>0$ and variance $\sigma_B^2>0$ depending on $B$ to ensure disconnection. 

As a first setup we consider $B=4$ with $S=6$ and $\sigma_B^2=1$ based on a partition of $[-1,1]^2$ into $2\times 2$ squares.
The numerical results are shown in Figure~\ref{fig:disjointMultimodalNumeric4Bumps}.
The second setup involves $B=8$ modes using a partition of $[-1,1]^2$ into $2\times 4$ squares with shift $S=6$ and smaller variance $\sigma_B^2=0.1$.

The results are visualized by means of the kernel density estimation in the \verb! seaborn ! package.
Motivated by Proposition~\ref{prop:mean}, the degrees of freedom associated to the (local) mean values are set as start value within the iterative minimizing process, while the remaining values are initialized randomly with standard normal distribution.
In particular, we set the degrees of freedom associated with polynomial coefficients for the constant contributions to the means $\mu_{S,\theta}$.
All other coefficients are randomly initialized \textit{iid} Gaussian with mean $0$ and variance $0.01$.

\begin{figure}
    \centering
    \includegraphics[width = 0.7\linewidth]{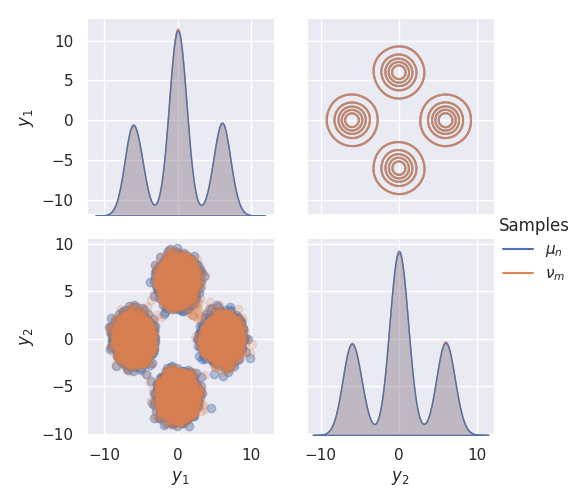}
    \caption{Verification of a multi-element polynomial chaos model class $\mathcal{M}$ with stochastic reference coordinate system $X=\mathcal{U}\left([-1,1]^2\right)$ with $2\times 2$ elements.
    The empirical measures are based on $n=m=\num{1.6e+4}$ samples. Each local polynomial is a tensor product of orthonormal polynomials of degree $7$.
    The final debiased Sinkhorn loss $\mathcal{S}_{\epsilon}$ with $\epsilon = 0.05$ is given by  $\mathcal{S}_{\epsilon}(\mu_n,\nu_m)\approx \num{6e-03}$.
    }
    \label{fig:disjointMultimodalNumeric4Bumps}
\end{figure}

\begin{figure}
    \centering
     \includegraphics[width = 0.7\linewidth]{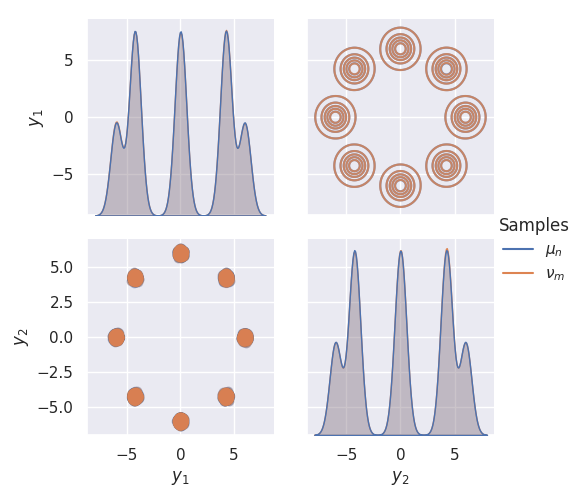}
    \caption{Verification of a multi-element polynomial chaos model class $\mathcal{M}$ with stochastic reference coordinate system $X=\mathcal{U}\left([-1,1]^2\right)$ with $2\times 4$ elements.
    The empirical measures are based on $n=m=\num{3.2e+4}$ samples.
    Each local polynomial is a tensor product of orthonormal polynomials of degree $7$.
    The final debiased Sinkhorn loss $\mathcal{S}_{\epsilon}$ with $\epsilon = 0.05$ is given by  $\mathcal{S}_{\epsilon}(\mu_n,\nu_m)\approx \num{2e-03}$.}
    \label{fig:disjointMultimodalNumeric8Bumps}
\end{figure}

\subsection{Random field with STT}
We consider the random field
\begin{equation}
\label{eq:randomField}
\kappa(x,\omega) = \cos(U_1(\omega) x + U_2(\omega)) + U_3(\omega),\quad x\in[0,1],
\end{equation}
with \textit{iid} random variables $U_1,U_2,U_3\sim\mathcal{U}(-1,1)$.
For $N\in\mathbb{N}$, let $x_i = (i-1)/(N-1)$, $i=1,\ldots,N$ be a  discretization of $[0,1]$ and define the $\mathbb{R}^N$-valued random variable $Y = (Y_i)_i$ by
$$
Y_i = \kappa(x_i,\omega),\quad i=1,\ldots, N.
$$
Figures~\ref{fig:rfN5} and~\ref{fig:rfN10} depict the correlation of the marginals for $N=5$ and $N=10$, respectively.
The plots show the different levels of correlation intensity, which corresponds to the distance of the discretization points in $[0,1]^2$.
The model class successfully captures these details with very small STT rank of $(2,2)$ with a debiased Sinkhorn loss of $\num{7.8e-4}$.

\begin{figure}
    \centering
    \includegraphics[width = 0.9\linewidth]{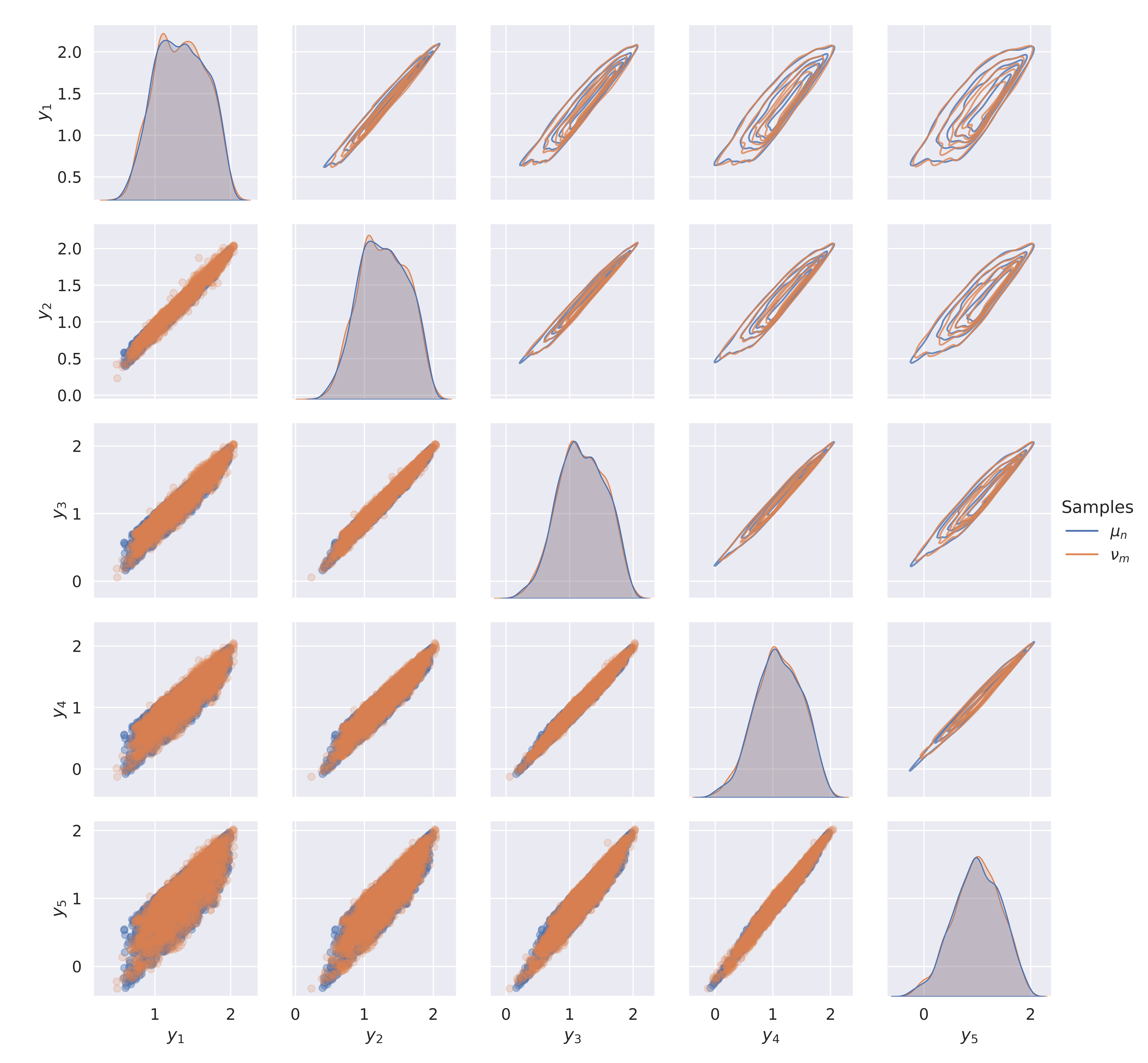}
    \caption{Marginals of the $N=5$ discretization points of the random field from \eqref{eq:randomField} with a STT of rank $r=(2,2)$ using tensorized Legendre polynomials of degree $(9,9,1)$ for $X\sim\mathcal{U}([-1,1]^3)$, \textit{i.e.} $M=3$ based on $n=m=\num{2e3}$ samples with $\epsilon = 0.05$ for the debiased Sinkhorn loss $\mathcal{S}_\epsilon(\mu_n,\nu_m)\approx\num{7.8e-4}$.}
    \label{fig:rfN5}
\end{figure}

\begin{figure}
    \centering
    \includegraphics[width=0.9\linewidth]{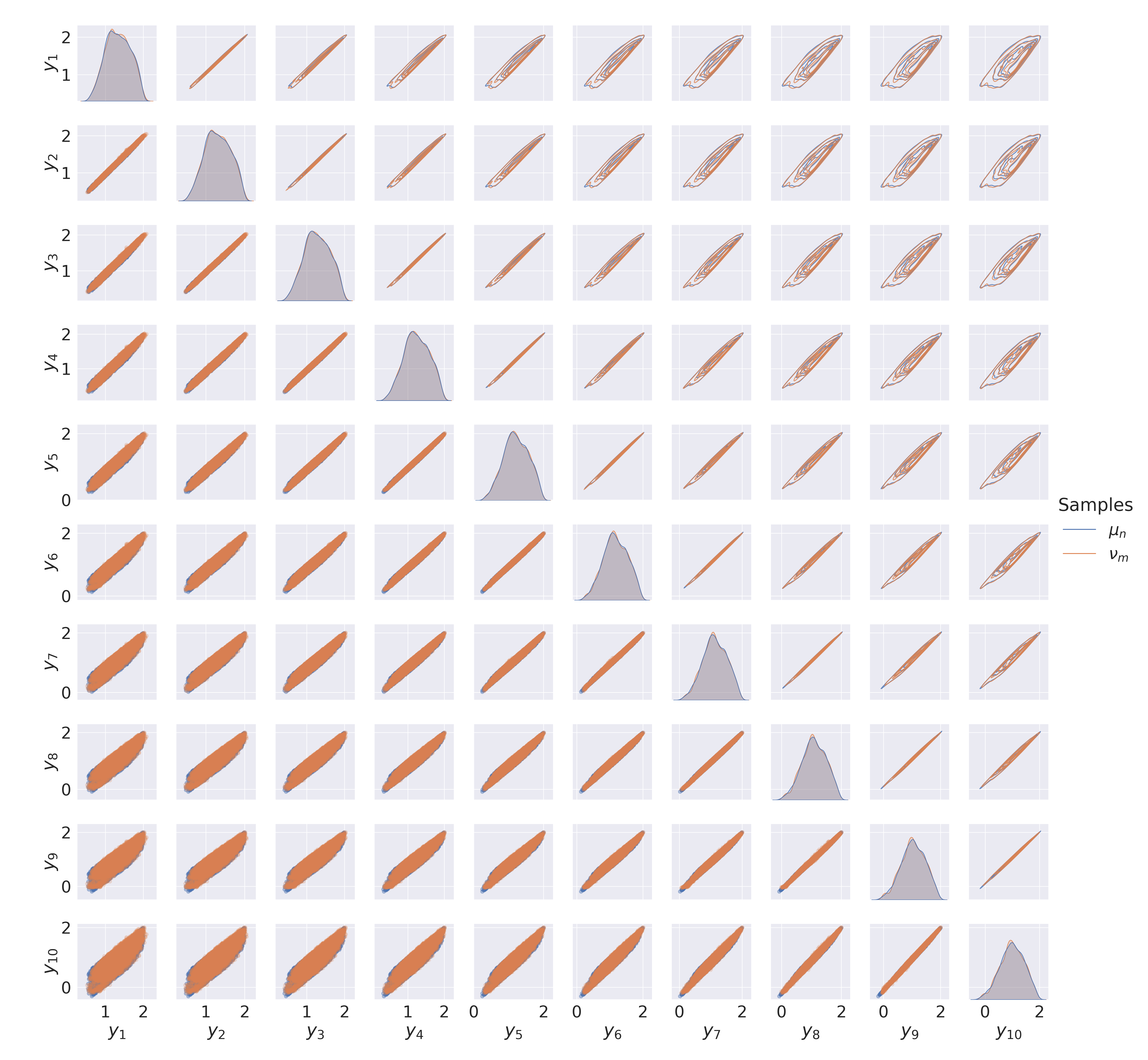}
    \caption{Marginals of the $N=10$ discretization points of the random field from \eqref{eq:randomField} with a STT of rank $r=(2,2)$ using tensorized Legendre polynomials of degree $(9,9,1)$ for $X\sim\mathcal{U}([-1,1]^3)$, \textit{i.e.} $M=3$ based on $n=m=\num{5e3}$ samples with $\epsilon = 0.05$ for the debiased Sinkhorn loss $\mathcal{S}_\epsilon(\mu_n,\nu_m)\approx\num{6.3e-4}$.}
    \label{fig:rfN10}
\end{figure}

\subsection{Application to numerical upscaling}
\label{sec:upscaling}

The motivation for the following application oriented experiment is the simulation with materials exhibiting random non-periodic microstructures as shown in Figure~\ref{fig:composite}.
Such computations pose significant challenges since the large number of inclusions results in a large set of stochastic dimensions.
This is due to the assumed parametrization, which determines the location as well as the shape of these inclusions.
For a statistical analysis of problems with such data, one would have to consider a large number of realizations for accurate results, which becomes computationally expensive due to the microscopic structures that have to be resolved.

An alternative way we pursue here is to determine an effective macroscopic (``upscaled'') random field which leads to the same statistical system response for some quantity of interest.
In particular, we use an approach of ``numerical upscaling'' to obtain (samples of) a macroscopic material description.
As a result, we derive a functional representation of a \textit{piecewise constant random field} as upscaled stochastic information.
Note that this is in contrast to constant effective material descriptions, which are usually obtained from (stochastic) homogenization methods where information about stochastic fluctuations is lost for subsequent numerical computations. 

The idea of functional representation of upscaled material is inspired by and extends the works in~\cite{soize2010identification} based on an optimization with maximum liklihood estimators (MLE) and~\cite{sarfaraz2018stochastic} where a functional representation is obtained by means of Kalman filters in a Bayesian context.
Here we obtain a functional representation by minimizing the debiased Sinkhorn divergence that metricizes the space of probability distributions.
As a result, based on the chosen model class we are able to fit random variables that are close to the observed samples in the sense of distribution.
This is fundamentally different from the functional representation obtained in~\cite{soize2010identification}, which uses a tensorized approximation of the MLE, relying on an approximation of uncorrelated random variables, and the representation in~\cite{sarfaraz2018stochastic}, where the resulting random variable has matching mean and variance only.

We consider a Lipschitz domain $D\subset\mathbb{R}^2$ on which a random composite field $\bm{\kappa} \in L^2(\Omega, \sigma, \mathbb{P}; L^\infty(D))$ is defined.
This random field represents a model for random composite materials.
Figure~\ref{fig:composite} depicts example realizations.
The random material is described as $2$-phase matrix composite material consisting of a matrix and inclusions.
Each inclusion is a star-shaped domain, \textit{i.e.} for $\kappa_0,\kappa_1\in\mathbb{R}_+$ 
$$\kappa(x, \omega) =
\left\{
\begin{array}{rcl}
\kappa_0 &=& x \in D_\text{incl}(\omega),\\
\kappa_1 &=& x \in D\setminus D(\omega), 
\end{array}
\right.
$$
for a random domain $D_\text{incl}(\omega)\subset D$ to be specified below. 
A sample of the random domain $D_\text{incl}$ is drawn as a set of star-shaped random inclusions, which are parametrized by their boundary description and correlated by geometric constraints.
Specifically, we assume that the set of inclusions is separated in the sense that the set of convex hulls of the polygonal approximations of the inclusions does not collide for a given prescribed discretization. 

Each realization of a random inclusion $\mathcal{I}(\omega)$ is given with an interface $\partial\mathcal{I}(\omega)$ parametrized as
$$
\partial\mathcal{I}(\omega) = \left\{P(\omega) +  \rho(\theta,\omega) ( \cos\theta, \sin\theta)^T, \theta\in[0,2\pi] \right\}\subset \mathbb{R}^2
$$
with center point $P \sim \mathcal{U}(D)$ and radius $\rho$ given as random field
$$\rho
(\theta, \omega) = \rho_0(\omega) \exp\left( \sum\limits_{\ell=1}^5 a_\ell(\omega) \sin(\ell \theta) + b_\ell(\omega) \cos(\ell \theta) \right)
$$
with $\rho_0\sim\mathcal{U}(0.01, 0.1)$ such that for the realization $\rho_0(\omega)$ we choose \textit{i.i.d.} $a_\ell,b_\ell\sim \mathcal{U}(-\rho_0(\omega), \rho_0(\omega))$ for $\ell=1,\ldots,5$.
Then a realization of $\kappa$ is obtained by successively adding or rejecting of up to $40$ inclusions that satisfy the geometric constraint. 

Subsequently, we consider a numerical upscaling scheme as used in~\cite{vasilyeva2021machine}.
Let $\delta D = \Gamma_{\mathrm{D}} \overset{\cdot}{\cup} \Gamma_{\mathrm{N}}$ with $|\Gamma_{\mathrm{D}}|>0$ for disjoint Dirichlet boundary $\Gamma_{\mathrm{D}}$ and Neumann boundary segment $\Gamma_{\mathrm{N}}$, respectively.
We consider a random partial differential equation given as
\begin{align}
\label{eq:darcyproblem}
\left\{
\begin{array}{rcll}
    -\operatorname{div} \mathcal{A}(x,\omega) \nabla \mathcal{U}(x,\omega) & = & f(x)  & \text{a.s. in } D\times \Omega, \\
                                                     \mathcal{U}(x,\omega) & = & g(x)  & \text{a.s. in } \Gamma_{\mathrm{D}}\times \Omega , \\
              \mathcal{A}(x,\omega) \partial_{\bm{n}}[\mathcal{U}](x,\omega) & = & h(x)  & \text{a.s. in }  \Gamma_{\mathrm{N}}\times \Omega,
\end{array}
\right.
\end{align}
with deterministic sufficiently regular data $f,g,h$ such that there exists a unique weak solution $\mathcal{U}(\cdot,\omega)$ $\mathbb{P}$-almost everywhere for a given realization of a random field $\mathcal{A}(\cdot,\omega)$ to be specified. 
Let $N_s\in\mathbb{N}$ and consider a disjoint decomposition $D= \bigcup_{s=1}^{N_s} D_s $.
For $\omega\in\Omega$ and on each $D_s$ choosing $\mathcal{A}=\kappa$, we solve the following auxiliary Dirichlet problems for $j=1,2$ on the micro scale,
\begin{align}
\label{eq:darcyproblem_micro}
\left\{
\begin{array}{rcll}
    -\operatorname{div} \mathcal{\kappa}(x,\omega) \nabla u_j(x,\omega) & = & 0  & \text{a.s. in } D_s\times \Omega, \\
                         u_j(x,\omega) & = & x_j  & \text{a.s. in } \partial D_s\times \Omega.
\end{array}
\right.
\end{align}
We refer to Figure~\ref{fig:composite} for an illustration of a partition for different realizations of the microscopic random field $\kappa$.
The effective macroscopic random permeability tensor field $K$ is assumed to be piecewise constant in $x$ with respect to the partition $\{D_s\}$, given by
$$
 [K(x,\omega)|_{D_s}]_{ij} \equiv K_{ij}^s(\omega):= \frac{1}{|D_s|}\int\limits_{D_s} \kappa(x,\omega)
 \frac{\partial u_i}{\partial x_j}\mathrm{d}x,\quad i,j = 1,2.
$$
Assuming that $K$ is symmetric and possibly anisotropic on $D_s$ for each $s=1,\ldots,N_s$, we can encode $K$ as a random vector $Y$ with values in $\mathbb{R}^{3N_s}$ and realization given by
\begin{equation}
\label{eq:Yupscaling}
Y(\omega) = [ K_{11}^1, K_{22}^1, K_{12}^1, \ldots, K_{11}^{N_s}, K_{22}^{N_s}, K_{12}^{N_s}](\omega)^T.
\end{equation}
While we can draw samples of $Y$ representing the upscaled random tensor field $K$ with the above concept, the actual distribution is unknown.
This motivates the functional representation of $Y$, which then enables to generate new samples very efficiently. 

For the numerical investigation we considered two cases with $N_s=1$ and $N_s=4$ in $D=[0,1]^2$ and partition as illustrated in Figure~\ref{fig:composite}.
The corresponding output random variable $Y$ from~\eqref{eq:Yupscaling} representing the anisotropic behaviour thus is $N=3$ and $N=12$ dimensional, respectively. 
\begin{figure}
    \centering
    \includegraphics[width = 0.9\linewidth]{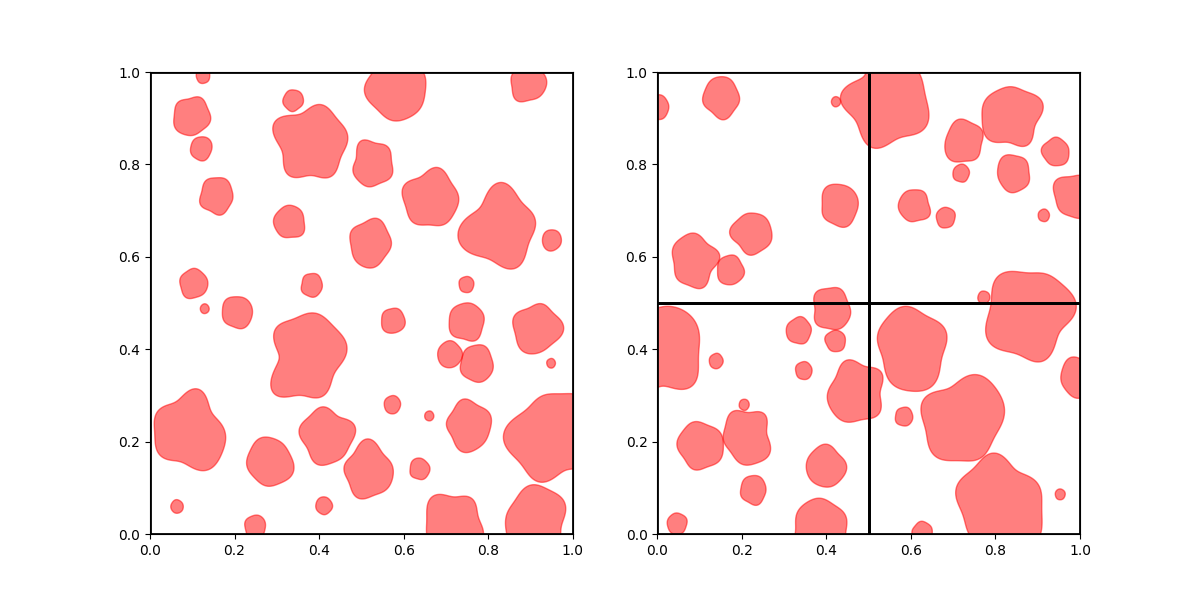}
    \caption{\textbf{Left:} Random composite sample with $N_s=1$ partition for the upscaling process. \textbf{Right:} Random composite sample with $N_s=4$ partition (black grid) used in the upscaling scheme.}
    \label{fig:composite}
\end{figure}

The associated coefficient tensor is represented as a STT with uniform rank $\bm{r}\equiv 6$.
The result of the fitting procedure with debiased Sinkhorn loss of $\num{2.4e-5}$ is shown in Figure~\ref{fig:upscalingsinkhornNs1} as corner plots. 
\begin{figure}
    \centering
    \includegraphics[width=0.9\linewidth]{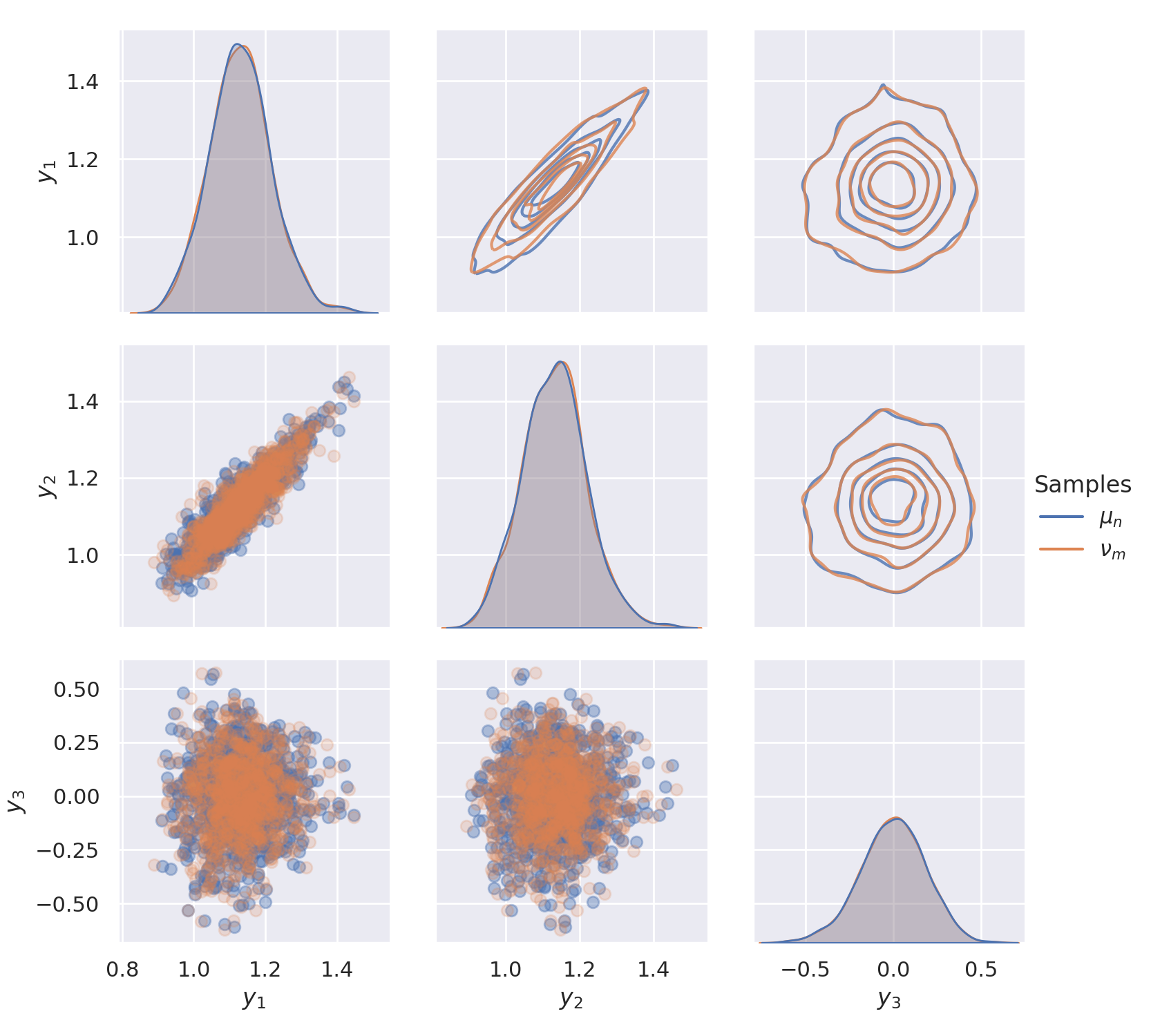}
    \caption{
    Upscaling experiment with $N_s=1$, $N=3$, $M=2$, polynomial degree $9$, and STT rank $\bm{r}=6$ with a resulting debiased Sinkhorn divergence of $\num{2.4e-5}$. Here, $Y1$ and $Y2$ are scaled with $0.1$ to improve the optimization performance.}
    \label{fig:upscalingsinkhornNs1}
\end{figure}

For the second case of $N_s=4$ with $\operatorname{img}(Y)\subset\mathbb{R}^{12}$, we use a stochastic reference coordinate system $X$ with $M=3$. 
Figure~\ref{fig:upscalingsinkhornNs4} depicts corner plots of the data samples of $Y$, showing strong correlation structure only locally that is associated with the diagonal entries of the upscaled anisotropic tensor.
To bridge the dimensional gap between $N=12$ and $M=3$, we allow flexible non-linear behaviour in terms of a large polynomial degree of $15$ and a STT rank of $(9,9)$.
The resulting final debiased Sinkhorn loss is about $\num{9e-3}$.
In contrast to the $N_s=1$ case, we observe a slightly larger deviation between the data and the model class distribution fit for the resulting marginals.

\begin{figure}
    \centering
    \includegraphics[width=0.95\linewidth]{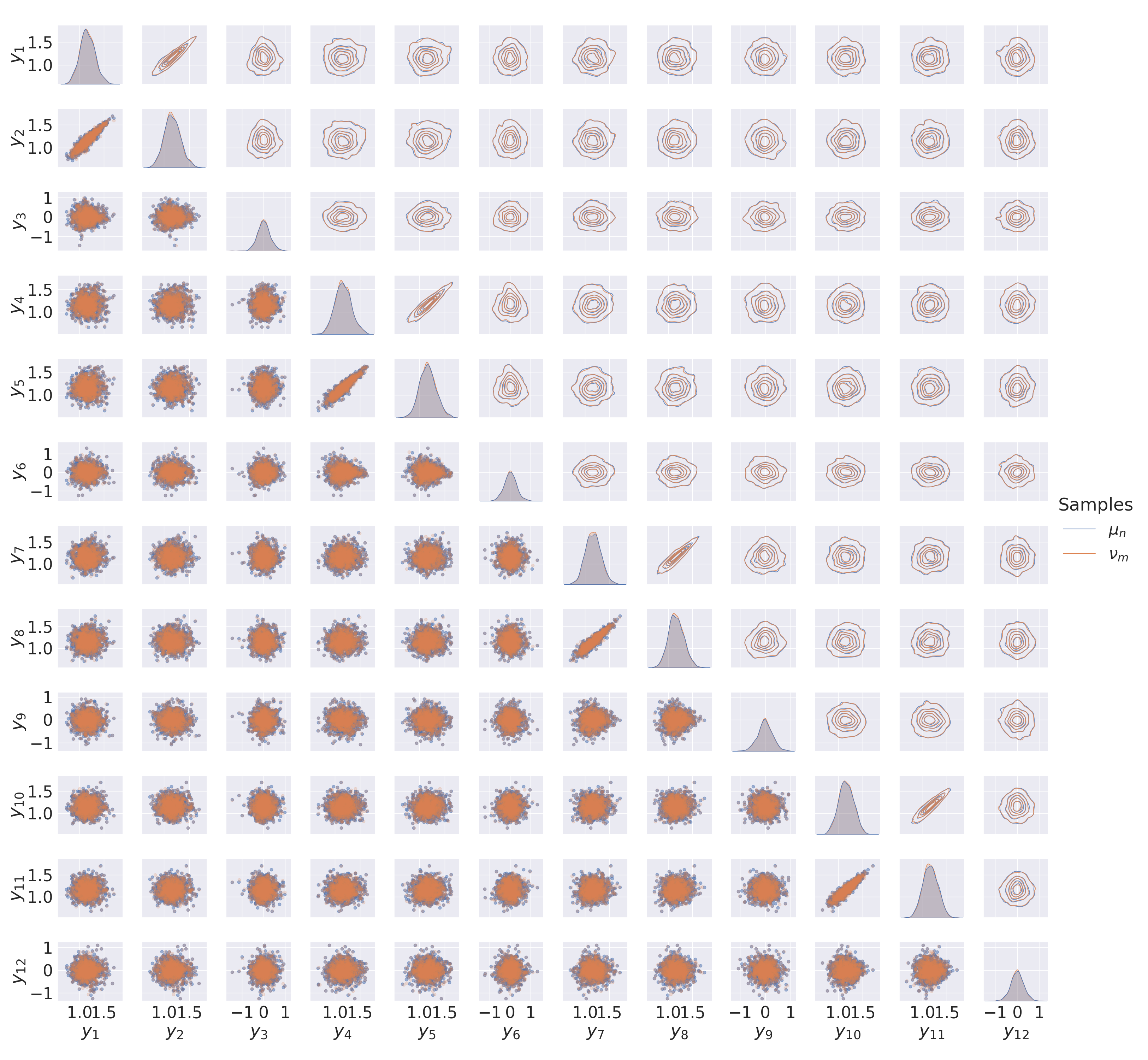}
    \caption{
    Upscaling experiment with $N_s=4$, $N=12$, $M=3$, polynomial degree $15$, and STT rank $(9,9)$ with a debiased Sinkhorn divergence of $\num{9e-3}$. $Y_k$ is scaled by $0.1$ for $k\mod 3 \equiv 1$ or $k\mod 3 \equiv 2$.}
    \label{fig:upscalingsinkhornNs4}
\end{figure}

Based on the obtained upscaled macroscopic random field $K$, we investigate the quality of the model class in terms of propagation and generalization.
For this, we consider the differential equation~\eqref{eq:darcyproblem} with $\mathcal{A} = K$.
Let $\mathcal{U}(\omega)$ denote the weak solution of~\eqref{eq:darcyproblem} for a realization of the random field $K(\omega)$.
Moreover, define the quantity of interest $q\colon H^1(D)\to \mathbb{R}$ as the spatial mean value
$$
    q(u) := \frac{1}{|D|}\int\limits_{D} v\mathrm{d}x,\quad \forall u\in H^1(D).
$$
This in turn defines the random variable $\overline{u}(\omega):=q(\mathcal{U}(\omega))$ that we use to assess the accuracy of the statistical properties of the computed macroscopic field.

We now propagate three different types of samples of $K$ through the Darcy problem (1) for one subdomain $N_s=1$ yielding $N=3$, and (2) for a decomposition into 4 subdomains $N_s=4$ yielding $N=12$, respectively.
First, the data samples of $Y$ are used to define samples of $K$ and thus samples of $\overline{u}$.
We refer to these samples as \textit{reference samples} propagated through $q$.
Second, the underlying samples of $X$ used to fit the model class $\mathcal{M}$ are propagated through $\mathcal{M}[\theta^\ast]$, yielding approximate samples of $Y$.
These are referred to as \textit{model fitted samples} propagated through $q$.
Third, new samples of the underlying reference coordinate system $X$ are drawn, mapped through the fit in the model class $\mathcal{M}[\theta^\ast]$ and eventually get propagated through $q$. 
Figure~\ref{fig:propagation} shows the results of these experiments.

On the one hand, we observe that the samples used to optimize the surrogate in the model class produce results that are very close to the propagated data sample distribution for both scenarios with $N_s=1$ and $N_s=4$, which is to be expected.
On the other hand, the propagation of newly generated samples results in a deviation of the resulting distribution for $N_s=4$, while there is only sample noise deviation for $N_s=1$.
Table~\ref{table:propagatedMoments1} and ~\ref{table:propagatedMoments2} depict the impact and strength of the deviation in terms of mean and variance.
A plausible explanation for this deviation is the very small number of samples used and the insufficient expressiveness of the model class $\mathcal{M}$ to ``explain'' the data sample distribution for $N_s=4$ with a resulting debiased Sinkhorn loss of $\num{2.9e-3}$ only, compared to the $\num{2.4e-5}$ loss value for $N_s=1$.

\begin{table}[ht]
\begin{tabular}{lccc}
\toprule
       & \multicolumn{3}{c}{samples}  \\
        \cmidrule{2-4}
&{\color{blue!50!white}{\Large\textbf{---}}}& 
{\color{orange}{\Large\textbf{---}}}&
{\color{green!50!black}{\Large\textbf{---}}}\\
\midrule
mean     &\num{7.41e-2} & \num{7.41e-2} & \num{7.43e-2}
      \\
variance     & \num{3.24e-5} &
     \num{3.18e-5} &
     \num{3.28e-5}\\
     \bottomrule
\end{tabular}
\centering
\caption{First and second order statistics of the $3$ resulting distributions displayed in Figure~\ref{fig:propagation}.  Case of one subdomain $N_s =1$.}
\label{table:propagatedMoments1}
\end{table}

\begin{table}[ht]
\begin{tabular}{lccc}
\toprule
       & \multicolumn{3}{c}{samples}  \\
        \cmidrule{2-4}
&{\color{blue!50!white}{\Large\textbf{---}}}& 
{\color{orange}{\Large\textbf{---}}}&
{\color{green!50!black}{\Large\textbf{---}}}\\
\midrule
mean     &\num{7.26e-2} & \num{7.25e-2} & \num{7.55e-2}
      \\
variance     & \num{3.75e-5} &
     \num{3.53e-5} &
     \num{6.11e-5}\\
     \bottomrule
\end{tabular}
\centering
\caption{First and second order statistics of the $3$ resulting distributions displayed in Figure~\ref{fig:propagation}. Case of $4$ subdomains $N_s=4$.}
\label{table:propagatedMoments2}
\end{table}

\begin{figure}
    \centering
    \includegraphics[width =0.49\linewidth]{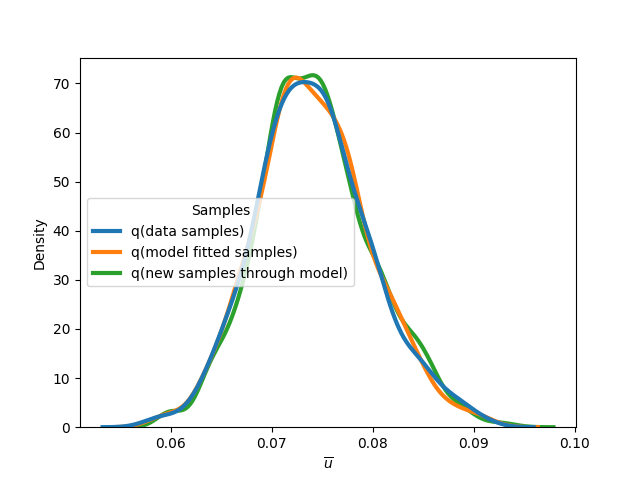}
    \includegraphics[width=0.49\linewidth]{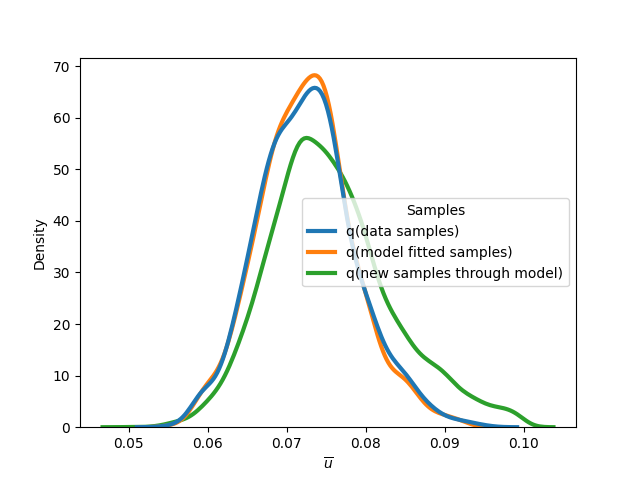}
    \caption{KDE plots of $q(\mathcal{Y})$ for three sets of samples of $\mathcal{Y}$:
    $\num{1e4}$ data samples of $Y$, $\num{1e4}$ propagated base samples through $\mathcal{M}$ and $\num{1e5}$ new samples drawn from $X$ again propagated through $\mathcal{M}$ are used to obtain samples from $\overline{u}$. \textbf{Left}: Case of one subdomain $N_s =1$. \textbf{Right}: Case of $4$ subdomains $N_s=4$.
    }
    \label{fig:propagation}
\end{figure}

\section{Conclusion}
\label{sec:conclusion}

In this work we consider a very flexible computational method to obtain functional representations of random variables that are determined by samples only without any knowledge of the underlying distribution.
The random variables (or fields) are expanded in a polynomial basis and optimized in an unsupervised way by means of debiased Sinkhorn losses that leads to a fitting of the discrete source and target measures.
In order to represent the uncertainty in the approximation, we use a multi-element polynomial chaos expansion as model class $\mathcal{M}$.
This high-dimensional representation is compressed with a possibly very efficient low-rank compression STT format.
The definition of (polynomial chaos) basis functions relies on the choice of the underlying stochastic input or reference coordinate system denoted as $X$, which is an $M$-dimensional random vector.
Since the target random vector $Y$ is $N$-dimensional, the model class $\mathcal{M}$ representation can be seen a relaxation of classical transport structures for which $N=M$ and $\mathcal{M}$ has to be a diffeomorphism.
We show such a relaxation in particular to discontinuous probability mass transport is necessary to obtain accurate functional representations of multi-modal measures.

While the optimization problem is non-convex and non-linear, the involved smoothness and efficient evaluation of the debiased Sinkhorn loss allows for fast computation of an accurate fit of measures through the parameter dependent model class in a distributional sense, based on second order schemes like BFGS or Hessian supported Newton schemes. 

The developed technique is applied to several common challenging tasks in UQ.
This includes the representation of a multi-modal distribution and an example of a smooth random field.
Furthermore, a numerical stochastic upscaling scheme with functional representation of the macroscale random field is developed.
This can be understood as a new approach within the framework laid out in~\cite{soize2010identification,sarfaraz2018stochastic,bigoni2019greedy}, where functional representations are obtained based on tensorized maximum likelihood losses, Kalman filtering in a Bayesian framework and optimization of Kullback-Leibner losses, respectively.
It extends these works in the sense that the new technique allows for a fitting in distribution, whereas only finite moment fits are obtained in~\cite{sarfaraz2018stochastic} or simplified approximations of maximum likelihood estimators of uncorrelated random variables are performed in~\cite{soize2010identification}, thus possibly ignoring any correlation structure between components of the samples.

As an outlook, the proposed unsupervised functional representation scheme can be applied and extended in various directions.
A first extension concerns the used metric $\mathrm{d}$ in the definition of the Wasserstein metric and in turn of the debiased Sinkhorn divergence.
While in this work we considered the Euclidean norm, choosing $\mathrm{d}$ as a Sobolev norm allows to learn random fields arising in random partial differential equations in distribution.
This idea extends the method presented in~\cite{eigel2019variational} based on empirical expectation losses.

A second extensions is related to Bayesian inference.
Here, we aim to obtain posterior information encoded in a functional approximation of a random variable as in~\cite{sarfaraz2018stochastic}, which corresponds to $N=M$ in our setting.
This representation allows the incoorperation of posterior information in follow-up computations based on samples or sampling free methods such as stochastic Galerkin schemes.

Third, we recall that the optimization subject to debiased Sinkhorn loss leads to a fit in the model class to the data sample distribution.
However, there might be circumstances where more statistical information of the image $Y$ such as moments is available.
This for instance arises with affine representations of random fields $\kappa$ in terms of Kosambi–Karhunen–Loève expansions
$$
\kappa(x,\omega) = \kappa_0(x) + \sum\limits_{i=1}^N Y_i(\omega) \kappa_i(x). 
$$
Here, the random vector $Y=(Y_1,\ldots,Y_N)$ has unknown distributions but is known to be centralized with uncorrelated $Y_i$ with unit variance, yielding a Stiefel manifold setting as in~\cite{soize2010identification}.
With this, the optimization problem~\eqref{eq:optprob} can easily be extended to a constrained formulation such that $\mathcal{M}[\theta]$ exhibits mean $0\in\mathbb{R}^N$ and covariance $I\in\mathbb{R}^{N,N}$.
In the performed numerical investigations, this constraint was introduced by penalty terms to circumvent difficulties that would arise when modifying the used STT format.
We believe that this type of additional information should improve the generalization error for small numbers of data samples for training in the model class.

Finally, the underlying debiased Sinkhorn divergence still is a dimension-dependent metric in terms of the sample complexity.
It would thus be worthwhile to extend our view to classes of dimension-free metrics such as the ones in the recent work~\cite{han2021class}.

\section*{Acknowledgment}
RG \& ME acknowledge support by the DFG MATH+ project AA5-5 (was EF1-25) -
\textit{Wasserstein Gradient Flows for Generalised Transport in Bayesian Inversion} and partial funding from the DFG SPP 1886 ``Polymorphic uncertainty modelling for the numerical design of structures''.
ME acknowledge partial funding from the DFG SPP 2298 ``Theoretical Foundations of Deep Learning'' and the ANR-DFG project ``COFNET: Compositional functions networks - adaptive learning for high-dimensional approximation and uncertainty quantification''.

\bibliographystyle{plain}
\bibliography{refs}

\begin{thebibliography}{10}

\bibitem{ali2020approximation}
Mazen Ali and Anthony Nouy.
\newblock Approximation with tensor networks. part ii: Approximation rates for
  smoothness classes.
\newblock {\em arXiv preprint arXiv:2007.00128}, 2020.

\bibitem{ali2021approximation}
Mazen Ali and Anthony Nouy.
\newblock Approximation with tensor networks. part iii: Multivariate
  approximation.
\newblock {\em arXiv preprint arXiv:2101.11932}, 2021.

\bibitem{altschuler2017near}
Jason Altschuler, Jonathan Weed, and Philippe Rigollet.
\newblock Near-linear time approximation algorithms for optimal transport via
  sinkhorn iteration.
\newblock {\em arXiv preprint arXiv:1705.09634}, 2017.

\bibitem{arjovsky2017wasserstein}
Martin Arjovsky, Soumith Chintala, and L{\'e}on Bottou.
\newblock Wasserstein generative adversarial networks.
\newblock In {\em International conference on machine learning}, pages
  214--223. PMLR, 2017.

\bibitem{bachmayr2021approximation}
Markus Bachmayr, Anthony Nouy, and Reinhold Schneider.
\newblock Approximation by tree tensor networks in high dimensions: Sobolev and
  compositional functions.
\newblock {\em arXiv preprint arXiv:2112.01474}, 2021.

\bibitem{bachmayr2016tensor}
Markus Bachmayr, Reinhold Schneider, and Andr{\'e} Uschmajew.
\newblock Tensor networks and hierarchical tensors for the solution of
  high-dimensional partial differential equations.
\newblock {\em Foundations of Computational Mathematics}, 16(6):1423--1472,
  2016.

\bibitem{baptista2020adaptive}
Ricardo Baptista, Olivier Zahm, and Youssef Marzouk.
\newblock An adaptive transport framework for joint and conditional density
  estimation.
\newblock {\em arXiv preprint arXiv:2009.10303}, 2020.

\bibitem{beck2014convergence}
Joakim Beck, Fabio Nobile, Lorenzo Tamellini, and Ra{\'u}l Tempone.
\newblock Convergence of quasi-optimal stochastic galerkin methods for a class
  of pdes with random coefficients.
\newblock {\em Computers \& Mathematics with Applications}, 67(4):732--751,
  2014.

\bibitem{behrmann2019invertible}
Jens Behrmann, Will Grathwohl, Ricky~TQ Chen, David Duvenaud, and
  J{\"o}rn-Henrik Jacobsen.
\newblock Invertible residual networks.
\newblock In {\em International Conference on Machine Learning}, pages
  573--582. PMLR, 2019.

\bibitem{blanc2007stochastic}
Xavier Blanc, Claude Le~Bris, and P-L Lions.
\newblock Stochastic homogenization and random lattices.
\newblock {\em Journal de math{\'e}matiques pures et appliqu{\'e}es},
  88(1):34--63, 2007.

\bibitem{blatman2008sparse}
G{\'e}raud Blatman and Bruno Sudret.
\newblock Sparse polynomial chaos expansions and adaptive stochastic finite
  elements using a regression approach.
\newblock {\em Comptes Rendus M{\'e}canique}, 336(6):518--523, 2008.

\bibitem{brauer2017sinkhorn}
Christoph Brauer, Christian Clason, Dirk Lorenz, and Benedikt Wirth.
\newblock A sinkhorn-newton method for entropic optimal transport.
\newblock {\em arXiv preprint arXiv:1710.06635}, 2017.

\bibitem{brennan2020greedy}
Michael Brennan, Daniele Bigoni, Olivier Zahm, Alessio Spantini, and Youssef
  Marzouk.
\newblock Greedy inference with structure-exploiting lazy maps.
\newblock {\em Advances in Neural Information Processing Systems},
  33:8330--8342, 2020.

\bibitem{bigoni2019greedy}
Michael Brennan, Daniele Bigoni, Olivier Zahm, Alessio Spantini, and Youssef
  Marzouk.
\newblock Greedy inference with structure-exploiting lazy maps.
\newblock {\em Advances in Neural Information Processing Systems},
  33:8330--8342, 2020.

\bibitem{chen2017tutorial}
Yen-Chi Chen.
\newblock A tutorial on kernel density estimation and recent advances.
\newblock {\em Biostatistics \& Epidemiology}, 1(1):161--187, 2017.

\bibitem{chkifa2015discrete}
Abdellah Chkifa, Albert Cohen, Giovanni Migliorati, Fabio Nobile, and Raul
  Tempone.
\newblock Discrete least squares polynomial approximation with random
  evaluations- application to parametric and stochastic elliptic pdes.
\newblock {\em ESAIM: Mathematical Modelling and Numerical Analysis},
  49(3):815--837, 2015.

\bibitem{cohen2015approximation}
Albert Cohen and Ronald DeVore.
\newblock Approximation of high-dimensional parametric pdes.
\newblock {\em Acta Numerica}, 24:1--159, 2015.

\bibitem{cohen2016expressive}
Nadav Cohen, Or~Sharir, and Amnon Shashua.
\newblock On the expressive power of deep learning: A tensor analysis.
\newblock In {\em Conference on learning theory}, pages 698--728. PMLR, 2016.

\bibitem{cohen2016convolutional}
Nadav Cohen and Amnon Shashua.
\newblock Convolutional rectifier networks as generalized tensor
  decompositions.
\newblock In {\em International Conference on Machine Learning}, pages
  955--963. PMLR, 2016.

\bibitem{cui2021deep}
Tiangang Cui and Sergey Dolgov.
\newblock Deep composition of tensor-trains using squared inverse rosenblatt
  transports.
\newblock {\em Foundations of Computational Mathematics}, pages 1--60, 2021.

\bibitem{cui2021conditional}
Tiangang Cui, Sergey Dolgov, and Olivier Zahm.
\newblock Conditional deep inverse rosenblatt transports.
\newblock {\em arXiv preprint arXiv:2106.04170}, 2021.

\bibitem{cuturi2013sinkhorn}
Marco Cuturi.
\newblock Sinkhorn distances: Lightspeed computation of optimal transport.
\newblock {\em Advances in neural information processing systems},
  26:2292--2300, 2013.

\bibitem{detommaso2018stein}
Gianluca Detommaso, Tiangang Cui, Youssef Marzouk, Alessio Spantini, and Robert
  Scheichl.
\newblock A stein variational newton method.
\newblock {\em Advances in Neural Information Processing Systems}, 31, 2018.

\bibitem{dolgov2019TTdensities}
Sergey Dolgov, Karim Anaya-Izquierdo, Colin Fox, and Robert Scheichl.
\newblock Approximation and sampling of multivariate probability distributions
  in the tensor train decomposition.
\newblock {\em Statistics and Computing}, 30(3):603--625, 11 2019.

\bibitem{dolgov2015polynomial}
Sergey Dolgov, Boris~N Khoromskij, Alexander Litvinenko, and Hermann~G
  Matthies.
\newblock Polynomial chaos expansion of random coefficients and the solution of
  stochastic partial differential equations in the tensor train format.
\newblock {\em SIAM/ASA Journal on Uncertainty Quantification},
  3(1):1109--1135, 2015.

\bibitem{dudley1969speed}
Richard~Mansfield Dudley.
\newblock The speed of mean glivenko-cantelli convergence.
\newblock {\em The Annals of Mathematical Statistics}, 40(1):40--50, 1969.

\bibitem{dunkl2014orthogonal}
Charles~F Dunkl and Yuan Xu.
\newblock {\em Orthogonal Polynomials of Several Variables}, volume 155.
\newblock Cambridge University Press, 2014.

\bibitem{dvurechensky2018computational}
Pavel Dvurechensky, Alexander Gasnikov, and Alexey Kroshnin.
\newblock Computational optimal transport: Complexity by accelerated gradient
  descent is better than by sinkhorn’s algorithm.
\newblock In {\em International conference on machine learning}, pages
  1367--1376. PMLR, 2018.

\bibitem{ENSW19}
M.~Eigel, J.~Neumann, R.~Schneider, and Sebastian Wolf.
\newblock Non-intrusive tensor reconstruction for high-dimensional random pdes.
\newblock {\em Computational Methods in Applied Mathematics}, 19:39 -- 53,
  2019.

\bibitem{eigel2014adaptive}
Martin Eigel, Claude~Jeffrey Gittelson, Christoph Schwab, and Elmar Zander.
\newblock Adaptive stochastic galerkin fem.
\newblock {\em Computer Methods in Applied Mechanics and Engineering},
  270:247--269, 2014.

\bibitem{alea}
Martin Eigel, Robert Gruhlke, Manuel Marschall, Philipp Trunschke, and Elmar
  Zander.
\newblock {ALEA} - {A} {P}ython {F}ramework for {S}pectral {M}ethods and
  {L}ow-{R}ank {A}pproximations in {U}ncertainty {Q}uantification.

\bibitem{EMPS20}
Martin Eigel, Manuel Marschall, Max Pfeffer, and Reinhold Schneider.
\newblock Adaptive stochastic galerkin {FEM} for lognormal coefficients in
  hierarchical tensor representations.
\newblock {\em Numerische Mathematik}, 145(3):655--692, 6 2020.

\bibitem{eigel2019non}
Martin Eigel, Johannes Neumann, Reinhold Schneider, and Sebastian Wolf.
\newblock Non-intrusive tensor reconstruction for high-dimensional random pdes.
\newblock {\em Computational Methods in Applied Mathematics}, 19(1):39--53,
  2019.

\bibitem{eigel2019variational}
Martin Eigel, Reinhold Schneider, Philipp Trunschke, and Sebastian Wolf.
\newblock Variational monte carlo—bridging concepts of machine learning and
  high-dimensional partial differential equations.
\newblock {\em Advances in Computational Mathematics}, 45(5):2503--2532, 2019.

\bibitem{ernst2012convergence}
Oliver~G Ernst, Antje Mugler, Hans-J{\"o}rg Starkloff, and Elisabeth Ullmann.
\newblock On the convergence of generalized polynomial chaos expansions.
\newblock {\em ESAIM: Mathematical Modelling and Numerical Analysis},
  46(2):317--339, 2012.

\bibitem{seaborn}
Michael~Waskom et~al.
\newblock seaborn: statistical data visualization.
\newblock {\em python package}, 2012-2021.

\bibitem{fenics}
{FEniCS Project - {A}utomated solution of {D}ifferential {E}quations by the
  {F}inite {E}lement {M}ethod}.

\bibitem{feydy2020geometric}
Jean Feydy.
\newblock {\em Geometric data analysis, beyond convolutions}.
\newblock PhD thesis, PhD thesis, Universit{\'e} Paris-Saclay, 2020.

\bibitem{feydy2019interpolating}
Jean Feydy, Thibault S{\'e}journ{\'e}, Fran{\c{c}}ois-Xavier Vialard, Shun-ichi
  Amari, Alain Trouv{\'e}, and Gabriel Peyr{\'e}.
\newblock Interpolating between optimal transport and mmd using sinkhorn
  divergences.
\newblock In {\em The 22nd International Conference on Artificial Intelligence
  and Statistics}, pages 2681--2690. PMLR, 2019.

\bibitem{genevay2019sample}
Aude Genevay, L{\'e}naic Chizat, Francis Bach, Marco Cuturi, and Gabriel
  Peyr{\'e}.
\newblock Sample complexity of sinkhorn divergences.
\newblock In {\em The 22nd International Conference on Artificial Intelligence
  and Statistics}, pages 1574--1583. PMLR, 2019.

\bibitem{genevay2018learning}
Aude Genevay, Gabriel Peyr{\'e}, and Marco Cuturi.
\newblock Learning generative models with sinkhorn divergences.
\newblock In {\em International Conference on Artificial Intelligence and
  Statistics}, pages 1608--1617. PMLR, 2018.

\bibitem{MR1083354}
Roger~G. Ghanem and Pol~D. Spanos.
\newblock {\em Stochastic finite elements: a spectral approach}.
\newblock Springer-Verlag, New York, 1991.

\bibitem{ghanem2003stochastic}
Roger~G Ghanem and Pol~D Spanos.
\newblock {\em Stochastic finite elements: a spectral approach}.
\newblock Courier Corporation, 2003.

\bibitem{gloria2014optimal}
Antoine Gloria, Stefan Neukamm, and Felix Otto.
\newblock An optimal quantitative two-scale expansion in stochastic
  homogenization of discrete elliptic equations.
\newblock {\em ESAIM: Mathematical Modelling and Numerical Analysis},
  48(2):325--346, 2014.

\bibitem{gloria2012optimal}
Antoine Gloria and Felix Otto.
\newblock An optimal error estimate in stochastic homogenization of discrete
  elliptic equations.
\newblock {\em The annals of applied probability}, pages 1--28, 2012.

\bibitem{goodfellow2014generative}
Ian Goodfellow, Jean Pouget-Abadie, Mehdi Mirza, Bing Xu, David Warde-Farley,
  Sherjil Ozair, Aaron Courville, and Yoshua Bengio.
\newblock Generative adversarial nets.
\newblock {\em Advances in neural information processing systems}, 27, 2014.

\bibitem{GM2024}
Robert Gruhlke and Dieter Moser.
\newblock Automatic differentiation within hierachical tensor formats.
\newblock {\em in preparation}, 2024.

\bibitem{bubbles}
Robert Gruhlke and Till Schäfer.
\newblock Bubbles - {A} {P}ython {F}ramework for composite modelling.

\bibitem{tensortrain}
Robert Gruhlke and David Sommer.
\newblock {TensorTrain} - {A} {P}ython {F}ramework for {T}ensor {T}rain
  approximations with {P}y{T}orch and {N}um{P}y backend.

\bibitem{hackbusch2012tensor}
Wolfgang Hackbusch.
\newblock {\em Tensor spaces and numerical tensor calculus}, volume~42.
\newblock Springer, 2012.

\bibitem{hamerly2004learning}
Greg Hamerly and Charles Elkan.
\newblock Learning the k in k-means.
\newblock {\em Advances in neural information processing systems}, 16:281--288,
  2004.

\bibitem{han2021class}
Jiequn Han, Ruimeng Hu, and Jihao Long.
\newblock A class of dimensionality-free metrics for the convergence of
  empirical measures.
\newblock {\em arXiv preprint arXiv:2104.12036}, 2021.

\bibitem{holtz2012manifolds}
Sebastian Holtz, Thorsten Rohwedder, and Reinhold Schneider.
\newblock On manifolds of tensors of fixed tt-rank.
\newblock {\em Numerische Mathematik}, 120(4):701--731, 2012.

\bibitem{kantorovich1942transfer}
Leonid Kantorovich.
\newblock On the transfer of masses (in russian).
\newblock In {\em Doklady Akademii Nauk}, volume~37, pages 227--229, 1942.

\bibitem{le2010spectral}
Olivier Le~Ma{\^\i}tre and Omar~M Knio.
\newblock {\em Spectral methods for uncertainty quantification: with
  applications to computational fluid dynamics}.
\newblock Springer Science \& Business Media, 2010.

\bibitem{lebrun2009rosenblatt}
R{\'e}gis Lebrun and Anne Dutfoy.
\newblock Do rosenblatt and nataf isoprobabilistic transformations really
  differ?
\newblock {\em Probabilistic Engineering Mechanics}, 24(4):577--584, 2009.

\bibitem{liu2019wasserstein}
Huidong Liu, Xianfeng Gu, and Dimitris Samaras.
\newblock Wasserstein gan with quadratic transport cost.
\newblock In {\em Proceedings of the IEEE/CVF international conference on
  computer vision}, pages 4832--4841, 2019.

\bibitem{luise2018differential}
Giulia Luise, Alessandro Rudi, Massimiliano Pontil, and Carlo Ciliberto.
\newblock Differential properties of sinkhorn approximation for learning with
  wasserstein distance.
\newblock {\em arXiv preprint arXiv:1805.11897}, 2018.

\bibitem{marzouk2016introduction}
Youssef Marzouk, Tarek Moselhy, Matthew Parno, and Alessio Spantini.
\newblock An introduction to sampling via measure transport.
\newblock {\em arXiv preprint arXiv:1602.05023}, 2016.

\bibitem{matthies2008stochastic}
Hermann~G Matthies.
\newblock Stochastic finite elements: Computational approaches to stochastic
  partial differential equations.
\newblock {\em ZAMM-Journal of Applied Mathematics and Mechanics/Zeitschrift
  f{\"u}r Angewandte Mathematik und Mechanik: Applied Mathematics and
  Mechanics}, 88(11):849--873, 2008.

\bibitem{mattila1999geometry}
Pertti Mattila.
\newblock {\em Geometry of sets and measures in Euclidean spaces: fractals and
  rectifiability}, volume~44.
\newblock Cambridge university press, 1999.

\bibitem{mensch2020online}
Arthur Mensch and Gabriel Peyr{\'e}.
\newblock Online sinkhorn: Optimal transport distances from sample streams.
\newblock {\em arXiv preprint arXiv:2003.01415}, 2020.

\bibitem{migliorati2013approximation}
Giovanni Migliorati, Fabio Nobile, Erik von Schwerin, and Ra{\'u}l Tempone.
\newblock Approximation of quantities of interest in stochastic pdes by the
  random discrete l\^{}2 projection on polynomial spaces.
\newblock {\em SIAM Journal on Scientific Computing}, 35(3):A1440--A1460, 2013.

\bibitem{morrison2017beyond}
Rebecca Morrison, Ricardo Baptista, and Youssef Marzouk.
\newblock Beyond normality: Learning sparse probabilistic graphical models in
  the non-gaussian setting.
\newblock {\em Advances in neural information processing systems}, 30, 2017.

\bibitem{nouy2015low}
Anthony Nouy.
\newblock Low-rank tensor methods for model order reduction.
\newblock {\em arXiv preprint arXiv:1511.01555}, 2015.

\bibitem{oseledets2011tensor}
Ivan~V Oseledets.
\newblock Tensor-train decomposition.
\newblock {\em SIAM Journal on Scientific Computing}, 33(5):2295--2317, 2011.

\bibitem{peyre2019computational}
Gabriel Peyr{\'e}, Marco Cuturi, et~al.
\newblock Computational optimal transport: With applications to data science.
\newblock {\em Foundations and Trends{\textregistered} in Machine Learning},
  11(5-6):355--607, 2019.

\bibitem{pinetz2018optimized}
Thomas Pinetz, Daniel Soukup, and Thomas Pock.
\newblock What is optimized in wasserstein gans?
\newblock In {\em Proceedings of the 23rd Computer Vision Winter Workshop},
  2018.

\bibitem{ramdas2017wasserstein}
Aaditya Ramdas, Nicol{\'a}s~Garc{\'\i}a Trillos, and Marco Cuturi.
\newblock On wasserstein two-sample testing and related families of
  nonparametric tests.
\newblock {\em Entropy}, 19(2):47, 2017.

\bibitem{sanjabi2018convergence}
Maziar Sanjabi, Jimmy Ba, Meisam Razaviyayn, and Jason~D Lee.
\newblock On the convergence and robustness of training gans with regularized
  optimal transport.
\newblock {\em Advances in Neural Information Processing Systems}, 31, 2018.

\bibitem{sarfaraz2018stochastic}
Sadiq~M Sarfaraz, Bojana~V Rosi{\'c}, Hermann~G Matthies, and Adnan
  Ibrahimbegovi{\'c}.
\newblock Stochastic upscaling via linear bayesian updating.
\newblock In {\em Multiscale Modeling of Heterogeneous Structures}, pages
  163--181. Springer, 2018.

\bibitem{schmitzer2019stabilized}
Bernhard Schmitzer.
\newblock Stabilized sparse scaling algorithms for entropy regularized
  transport problems.
\newblock {\em SIAM Journal on Scientific Computing}, 41(3):A1443--A1481, 2019.

\bibitem{schrodinger1932theorie}
Erwin Schr{\"o}dinger.
\newblock Sur la th{\'e}orie relativiste de l'{\'e}lectron et
  l'interpr{\'e}tation de la m{\'e}canique quantique.
\newblock In {\em Annales de l'institut Henri Poincar{\'e}}, volume~2, pages
  269--310, 1932.

\bibitem{schwab2011sparse}
Christoph Schwab and Claude~Jeffrey Gittelson.
\newblock Sparse tensor discretizations of high-dimensional parametric and
  stochastic pdes.
\newblock {\em Acta Numerica}, 20:291--467, 2011.

\bibitem{soize2010identification}
Christian Soize.
\newblock Identification of high-dimension polynomial chaos expansions with
  random coefficients for non-gaussian tensor-valued random fields using
  partial and limited experimental data.
\newblock {\em Computer methods in applied mechanics and engineering},
  199(33-36):2150--2164, 2010.

\bibitem{spantini2018inference}
Alessio Spantini, Daniele Bigoni, and Youssef Marzouk.
\newblock Inference via low-dimensional couplings.
\newblock {\em The Journal of Machine Learning Research}, 19(1):2639--2709,
  2018.

\bibitem{steinlechner2016riemannian}
Michael~Maximilian Steinlechner.
\newblock Riemannian optimization for solving high-dimensional problems with
  low-rank tensor structure.
\newblock Technical report, EPFL, 2016.

\bibitem{szalay2015tensor}
Szil{\'a}rd Szalay, Max Pfeffer, Valentin Murg, Gergely Barcza, Frank
  Verstraete, Reinhold Schneider, and {\"O}rs Legeza.
\newblock Tensor product methods and entanglement optimization for ab initio
  quantum chemistry.
\newblock {\em International Journal of Quantum Chemistry}, 115(19):1342--1391,
  2015.

\bibitem{varadarajan1958convergence}
Veeravalli~S Varadarajan.
\newblock On the convergence of sample probability distributions.
\newblock {\em Sankhy{\=a}: The Indian Journal of Statistics (1933-1960)},
  19(1/2):23--26, 1958.

\bibitem{vasilyeva2021machine}
Maria Vasilyeva and Aleksey Tyrylgin.
\newblock Machine learning for accelerating macroscopic parameters prediction
  for poroelasticity problem in stochastic media.
\newblock {\em Computers \& Mathematics with Applications}, 84:185--202, 2021.

\bibitem{villani2009optimal}
C{\'e}dric Villani.
\newblock {\em Optimal transport: old and new}, volume 338.
\newblock Springer, 2009.

\bibitem{wan2006multi}
Xiaoliang Wan and George~Em Karniadakis.
\newblock Multi-element generalized polynomial chaos for arbitrary probability
  measures.
\newblock {\em SIAM Journal on Scientific Computing}, 28(3):901--928, 2006.

\bibitem{weed2019sharp}
Jonathan Weed and Francis Bach.
\newblock Sharp asymptotic and finite-sample rates of convergence of empirical
  measures in wasserstein distance.
\newblock {\em Bernoulli}, 25(4A):2620--2648, 2019.

\bibitem{werner2016positive}
Albert~H Werner, Daniel Jaschke, Pietro Silvi, Martin Kliesch, Tommaso Calarco,
  Jens Eisert, and Simone Montangero.
\newblock Positive tensor network approach for simulating open quantum
  many-body systems.
\newblock {\em Physical review letters}, 116(23):237201, 2016.

\bibitem{xiu2002wiener}
Dongbin Xiu and George~Em Karniadakis.
\newblock The wiener--askey polynomial chaos for stochastic differential
  equations.
\newblock {\em SIAM journal on scientific computing}, 24(2):619--644, 2002.

\bibitem{zech2021sparse}
Jakob Zech and Youssef Marzouk.
\newblock Sparse approximation of triangular transports. part ii: the infinite
  dimensional case.
\newblock {\em arXiv preprint arXiv:2107.13422}, 2021.

\end{thebibliography}

\end{document}